\documentclass[11pt]{article}

\usepackage{float}
\usepackage{subfigure}
\usepackage{graphicx}
\usepackage{amsmath}
\usepackage{amssymb}
\usepackage{amsthm}
\usepackage{amsfonts}
\usepackage{enumerate}
\usepackage{mathtools}
\usepackage{cite}
\usepackage{tikz}
\usepackage{tabularx}
\usepackage{pgfplots}
\usepackage{float}
\usepackage{booktabs} 
\usepackage{microtype} 
\usepackage{amsfonts}
\usepackage[english]{babel}
\usepackage{caption}
\usepackage{multicol}
\usepackage{gauss} 
\usepackage{etoolbox} 
\usepackage{geometry}
\usepackage{todonotes}
\usepackage{epigraph}
\usepackage[perpage,para,symbol*]{footmisc}
\usepackage{url}

\newcolumntype{x}[1]{!{\centering\arraybackslash\vrule width #1}}

\newtheorem{theorem}{Theorem}
\newtheorem{lemma}{Lemma}
\newtheorem{remark}{Remark}
\newtheorem{corollary}{Corollary}
\newtheorem{definition}{Definition}
\newtheorem{conjecture}{Conjecture}
\newtheorem{proposition}{Proposition}

\newcommand{\R}{\mathbb{R}}
\newcommand{\N}{\mathbb{N}}

\definecolor{edgecolour}{RGB}{39,64,112}
\definecolor{dualvertexcolour_pentagon}{RGB}{205,79,57}
\definecolor{dualvertexcolour_hexagon}{RGB}{79,205,57}
\tikzstyle{pentagon}=[circle,draw,inner sep=0pt, fill=dualvertexcolour_pentagon, minimum width=1.5mm]
\tikzstyle{hexagon}=[circle,draw,inner sep=0pt, fill=dualvertexcolour_hexagon, minimum width=1.5mm]
\tikzstyle{facet}=[circle,draw,inner sep=0pt, minimum width=1.5mm]
\tikzstyle{edge}=[draw, color=edgecolour, line width=0.3mm]

\makeatletter
\let\@fnsymbol\@arabic
\newcommand{\specificthanks}[1]{\@fnsymbol{#1}}
\makeatother

\begin{document}

\title{\textbf{Spectral clustering of combinatorial fullerene isomers based on their facet graph structure}}
\insert\footins{\footnotesize \noindent Artur Bille \\ \texttt{artur.bille@uni-ulm.de}\vspace*{0.2cm} \\  Victor Buchstaber\\ \texttt{buchstab@mi-ras.ru}\vspace*{0.2cm} \\Evgeny Spodarev\\ \texttt{evgeny.spodarev@uni-ulm.de}\vspace*{0.5cm}}

\author{Artur Bille\thanks{Ulm University} \textsuperscript{,}\thanks{Skoltech}\and 
Victor Buchstaber\textsuperscript{\specificthanks{2},}\thanks{Steklov Mathematical Institute}\and 
Evgeny Spodarev\textsuperscript{\specificthanks{1}}
}
\date{}
\maketitle

\epigraph{... This spiritual experience, this discovery of what Nature has in store for us with carbon, is still ongoing.}{Richard E. Smalley, {\it Discovering the fullerenes}, Nobel lecture, Dec. 7, 1996.}

\begin{abstract}
\label{abstract}
After Curl, Kroto and Smalley were awarded 1996 the Nobel Prize in chemistry, fullerenes have been subject of much research. One part of that research is the prediction of a fullerene's stability using topological descriptors. It was mainly done by considering the distribution of the twelve pentagonal facets on its surface, calculations mostly were performed on all isomers of $C_{40}, C_{60}$ and $C_{80}$.
This paper suggests a novel method for the classification of combinatorial fullerene isomers using spectral graph theory. The classification presupposes an invariant scheme for the facets based on the Schlegel diagram. The main idea is to find clusters of isomers by analyzing their graph structure of hexagonal facets only. We also show that 
our classification scheme can serve as a formal stability criterion, which became evident from a comparison of our results with recent quantum chemical calculations \cite{Grimme17}.
We apply our method to classify all isomers of $C_{60}$ and give an example of two different cospectral isomers of $C_{44}$.\\
Calculations are done with MATLAB. The only input for our algorithm is the vector of positions of pentagons in the facet spiral. These vectors and Schlegel diagrams are generated with the software package Fullerene \cite{PeSchwertFull}.
\end{abstract}


\section{Introduction}
\label{sec: Introduction}
In this paper, we consider a fullerene $C_n$ as a convex polytope modeling a closed three-dimensional carbon-cage with $n$ atoms, cf. \cite{AndKarSkr16}. Each vertex is connected to exactly three other vertices, such that the facets are pentagons and hexagons only. Using the classical Euler relation and Eberhard's theorem \cite[\S 13.3]{Grunbaum} for simple three-dimensional convex polytopes one can conclude that the number of pentagonal facets is always equal to $12$ and $n$ is even. In this case, the number of hexagons is $m_6=\frac{n}{2}-10$, cf. \cite{FowlerManop}. Theoretically, $C_{n}$ exists for $n=20$ and all even $n\geq 24$, see \cite{AndKarSkr16}. In the sequel we call such a $n$ \textit{feasible}.  However, up to now 

\begin{figure}[H]    								
\centering 	
\subfigure{									
\includegraphics[scale=0.3]{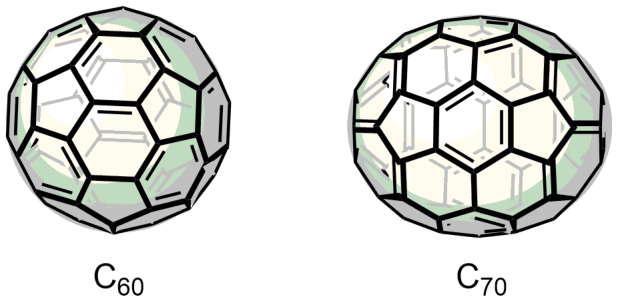}
}
\caption{IPR-Isomers of $C_{60}$ (Buckminster fullerene) and $C_{70}$. Courtesy of Max von Delius}
\label{fig:IPR}
\end{figure}

only a few of them (some isomers of $C_{n}$ with $n=60,70,76,78,80,82,84$) are separated (i.e., are chemically stable and can be synthesized in considerable mass quantities), cf. \cite{ZieglerAmsharov11,Mojicaetal13,MarPao14,Hirsch_Bettreich_Wudl}. Another difficult problem is the combinatorial constructive  enumeration of all isomers of $C_n$, see \cite{BriDre,Thurs98,BuchEro17} and references in \cite[Section 3]{Erokh18}. For instance, a set of four operations (including that of Endo-Kroto) enables the construction of all fullerenes for each even $n$ starting with $n=24$ (barrel), see \cite{BuchstEro,Erokh18}. Other examples for such operations are the {\it generalized Stone-Wales operation} (cf. \cite{generalized_SW}, which we investigate in a forthcoming paper \cite{BBS_2} in more detail) and the {\it buckygen} (introduced 2012 in \cite{buckygen})  which is up to now the fastest algorithm to generate all $C_n$-isomers.

Combinatorial isomer is a class of combinatorially equivalent polytopes. Two polytopes $P$ and $P'$ are called \textit{combinatorially equivalent} if there exists a one-to-one mapping between the lattice of all facets of $P$ and $P'$ that is inclusion-preserving \cite[p. 38]{Grunbaum}. Hence, this equivalence class is determined by the graph of vertices, or equivalently, by the dual graph of facets. Both are uniquely described by their adjacency matrices. With increasing amount of atoms $n$, the number of isomers ISO$\left(C_n\right)$ grows fast as  $O(n^9)$, cf. \cite{Thurs98}. For instance, there exist a unique $C_{20}$--isomer (we write: $|C_{20}|=1$) which is {\it dodecahedron}, a Platonic solid, but $C_{60}$ has already $1812$ combinatorial isomers, see \cite{HouseofGraphs}. Among all isomers of $C_{60}$, only one (the so--called {\it Buckminster} fullerene, an Archimedean truncated icosahedron) has all pentagonal facets being not adjacent. Such isomers are called {\it IPR-fullerenes} (from {\it Isolated Pentagon Rule}), cf. Figure \ref{fig:IPR}. As illustrated in Figure \ref{fig:number_isomere}, the number ISO-IPR$\left(C_n\right)$ of IPR-isomers also grows asymptotically as $O(n^9)$  with increasing number of atoms $n$ \cite{BriDre,Rukh18}. 

The huge variety of possible fullerene isomers  with large $n$ (even in the IPR--class) makes the problem of finding molecules with remarkable chemical and physical attributes (including thermodynamic and kinetic stability, permeability, electric conductivity, light diffraction, etc.) extremely difficult. Hence, a need for fast and computationally cheap classification methods of isomers arises. In the literature, there already exist a number of functionals (called {\it chemical descriptors}) allowing to find a certain order of isomers, cf. \cite{HAYAT2018164,Baca,MT}. These descriptors are of topological, geometric or physical nature. For the practical separation of fullerenes,  their potential energetic level is of primary significance. The paper \cite{Grimme17}  computes the relative energies of all $1812$ isomers of $C_{60}$ using the density functional theory (DFT) \cite{DFT} and testing $26$ chemical descriptors for their correlation with the energetic ordering of these isomers. The authors identify $7$ rules among $26$ which they call  {\it good stability criteria}. These criteria are defined as those able to identify correctly the first two energetically most stable  and the three energetically least stable isomers in the correct energetic order such that the correlation coefficient between the ordering according to these rules and the energetic order is at least $0.6$. However, the DFT calculations are based on the approximative numerical solution of  electronic Schr\"odinger (linear partial differential) equations requiring from half an hour up to one day of calculation time per isomer ($n=60$) on a usual personal computer. Moreover, the convergence of the DFT numerical method is not guaranteed.    

Following the famous question by Mark Kac (1966) {\it Can one hear the shape of a drum?} \cite{Kac66} we try to ``hear'' the shape of a fullerene from the spectrum of its adjacency matrix provided that $C_n$--isomers can be mapped bijectively onto their spectra. We also give an example of two different $C_{44}$--isomers with the same spectrum of adjacency matrix of hexagonal facets.
We propose a new method of clustering and of classification of $C_n$--isomers for any $n\ge 24$, $n\neq 44$, based on combinatorial and graph theoretic structure of dual graphs $T^6_{n}$  of their hexagonal facets. For $n=60$ we show in this paper that it yields a good stability criterion with correlation coefficient of $0.9$.
The spectral analysis of graphs based on adjacency matrices of their vertices has a long standing tradition \cite{BroHae,BruCve,Bap}.
However, we show that for the complete classification and ordering of fullerenes it is sufficient to use
\begin{itemize}
\item  the dual graph $T^6_{n}$ of hexagons since the positions of 12 pentagons can be reconstructed out of cycles larger than triangles and degrees of vertices in $T^6_n$  (cf. \cite{BBS_2}).
\item Newton polynomials of the spectrum of the adjacency matrix $A^6_{n}$ of hexagons up to a certain degree $k^*$ since it is well known that they are numerically more stable than the eigenvalues themselves. These polynomials can be computed directly as a trace of $\left(A^6_{n}
\right)^k$, $k=2q$, $q=1,\ldots, k^*/2$, whereas the matrix multiplication is computationally less demanding than finding all eigenvalues of a matrix. Hereby, we use a graph theoretical interpretation of Newton polynomials of adjacency matrices of graphs in terms of their cycle numbers.   
\end{itemize}
Our  method is computationally very fast requiring $O(n^3\log n)$ operations with a total of  1.15s CPU time on a Intel Core i5-8300H (2.3 GHz) ($n=60$). The high correlation of the obtained ordering with the DFT energetic order allows to figure out few energetically stable isomers at a low computational cost.  For these isomer candidates, the detailed DFT analysis can be further performed. 

\begin{figure}[H]    								
\centering 	
\subfigure{									
\includegraphics[scale=0.3]{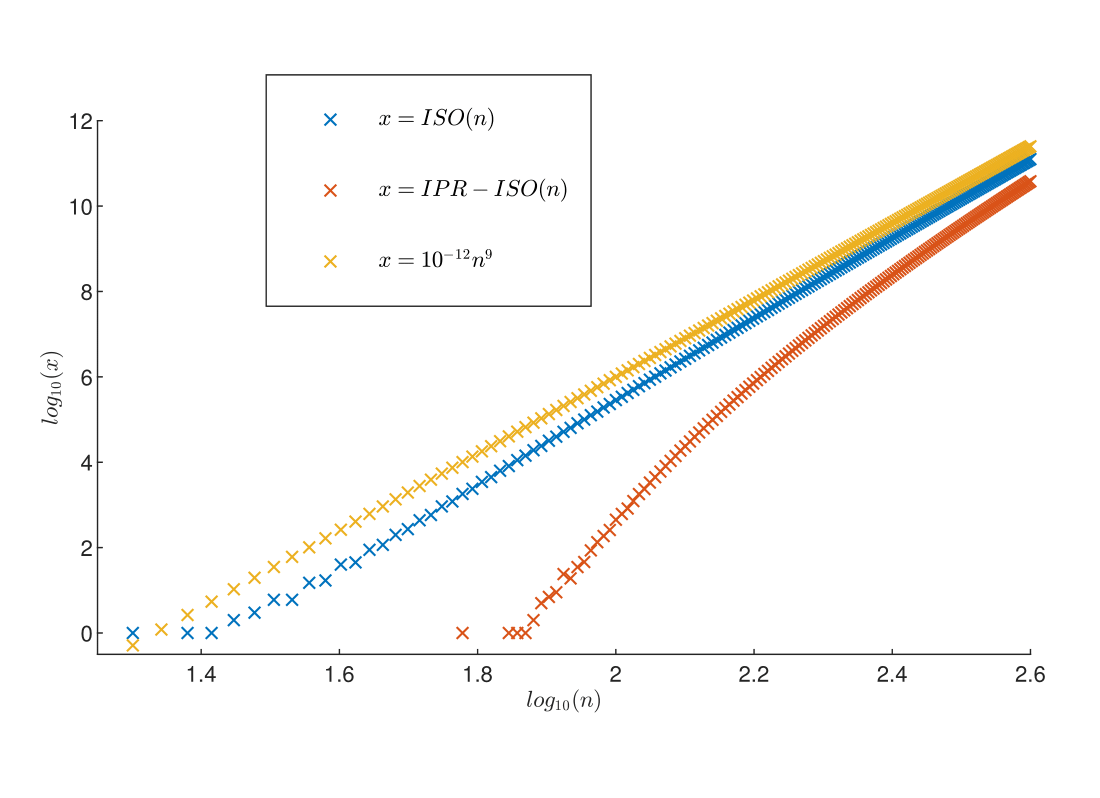}
}
\caption{Logarithm of the numbers of $C_n$-isomers $ISO(n)$, IPR-isomers $IPR-ISO(n)$ and their upper bound for all feasible $n\in\lbrace20,\ldots,400 \rbrace$, on a  logarithmic scale.}
\label{fig:number_isomere}
\end{figure}

%

In order to construct the facet adjacency matrices, a certain enumeration algorithm of all facets is required. Since the spectra of these matrices are invariant with respect to the enumeration of facets, the choice of this algorithm does not matter from the mathematical point of view. For all fullerenes with $24\le n < 380$, we use the {\it spiral rule} first introduced in \cite{ManMayDown91}  (where it is called an {\it orange peel} scheme) to enumerate all pentagons and hexagons and generate their adjacency matrices $A_{n}^5$ and $A_{n}^6$, respectively. The first fullerene not obeying the spiral rule is a $C_{380}$--isomer,  and the second counterexample is one of over 90 billion $C_{384}$-isomers. All other isomers of $C_n$ with $n\leq 450$ stick to this rule, cf. \cite{ManFow}. For fullerenes without a facet spiral, a generalized spiral \cite{WSA17} can be used for the one-to-one facet enumeration.


 Our spectral approach is illustrated on all isomers of $C_{60}$ which are the best studied fullerenes, especially the above mentioned famous {\it Buckminster} (soccer ball-like molecule). Such molecular structures are all allotropic forms of carbon \cite{MarPao14}. 

\section{Spectral analysis of  $\boldsymbol{C_{n}}$}
\label{sec:Spectral}
For a feasible $n$ a {\it fullerene isomer} $P\in C_n$ is a simple, compact and convex polytope in $\mathbb{R}^3$ with all $m:=n/2+2$ facets being either one of 12 pentagons or one of $n/2-10$ hexagons:
\begin{equation*}
P := \{ x \in \mathbb{R}^3 ~|~ a_i x+b_i \geq 0\; i=1,\ldots, m\}, \quad a_i\neq 0, \; b_i\in \R, \mbox{ for all }i .
\end{equation*}
 Its $i^\text{th}$ facet $f_i$ is given by
$
f_i := \{ x\in \mathbb{R}^3 ~|~ a_i x + b_i = 0\}\cap P$ , $i=1,\ldots , m.$
$P$ can be mapped on a two--dimensional graph in a way that edge crossing is avoided and vertex connectivity information is retained. First, one has to choose a facet and rotate $P$ so that this facet is located parallel to the $(x,y)$--plane at some distance below a fixed projection point $q$.
Next, one draws a line starting in $q$ to each vertex of the polyhedron and extends this line until it crosses the $(x,y)$--plane. The intersections are the vertices of the new two-dimensional graph, also called \textit{Schlegel diagram}. Although such a projection is not bijective, it yields a full combinatorial invariant of $P$. Each of $f_i$ can be chosen to be initially parallel to the $x-y$--plane. Depending on this choice, the resulting graphs can be very different, see \cite{FowlerManop}. In Figure \ref{fig:buckyball} and \ref{fig:buckyball_diff}, one can see two possible Schlegel diagrams for the Buckminster fullerene ($n=60$). The graphs  on these Schlegel diagrams are equivalent in the sense that they have the same vertex connectivity. From the definition of a fullerene it immediately follows that the corresponding planar graph is 3-regular. We denote a planar graph of a fullerene by $F$.

\begin{figure}[H]
\centering
\subfigure[a pentagon was chosen initially]{
\includegraphics[scale=0.15]{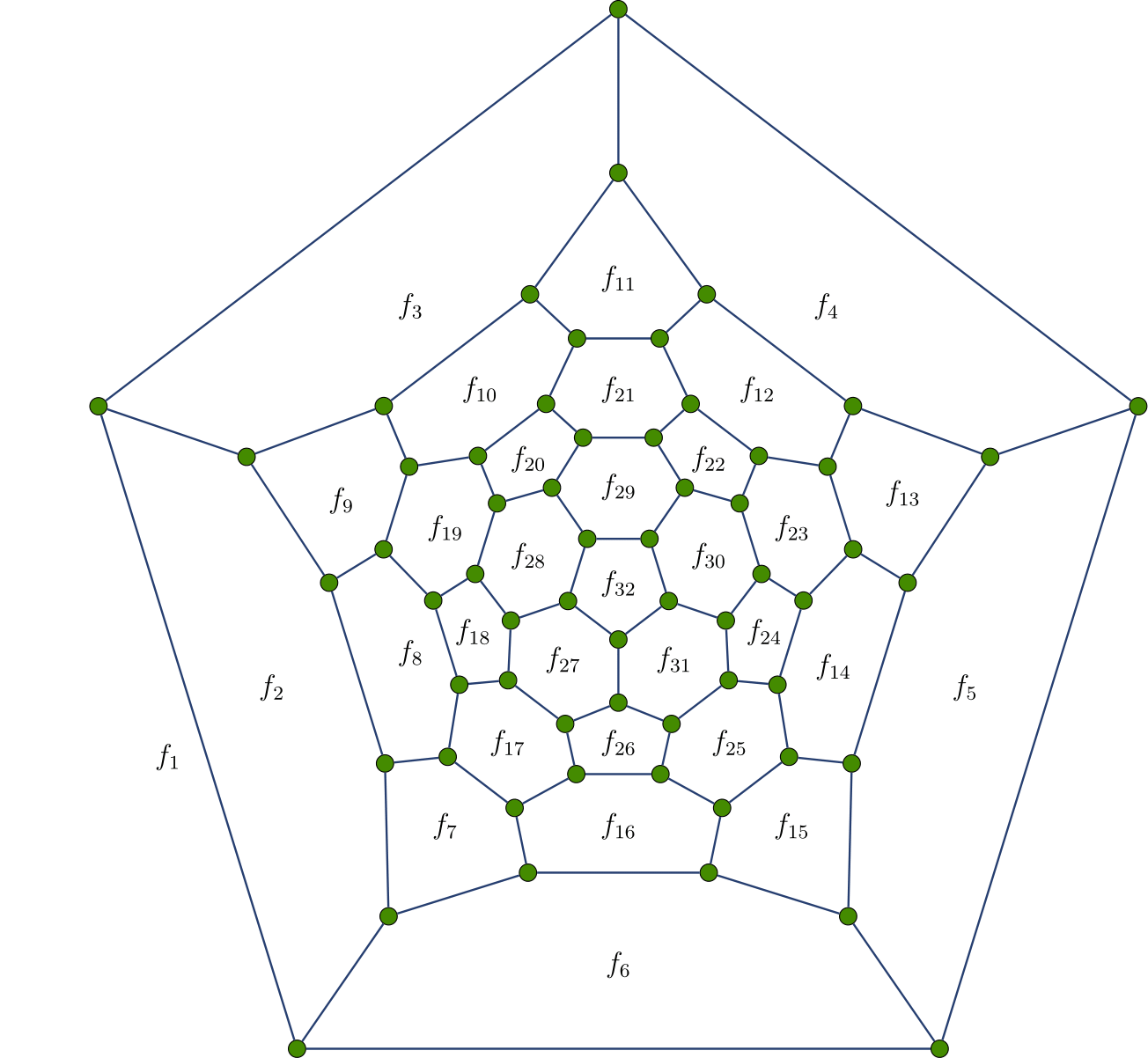}
\label{fig:buckyball}
} 
\subfigure[a hexagon was chosen initially]{
\includegraphics[scale=0.15]{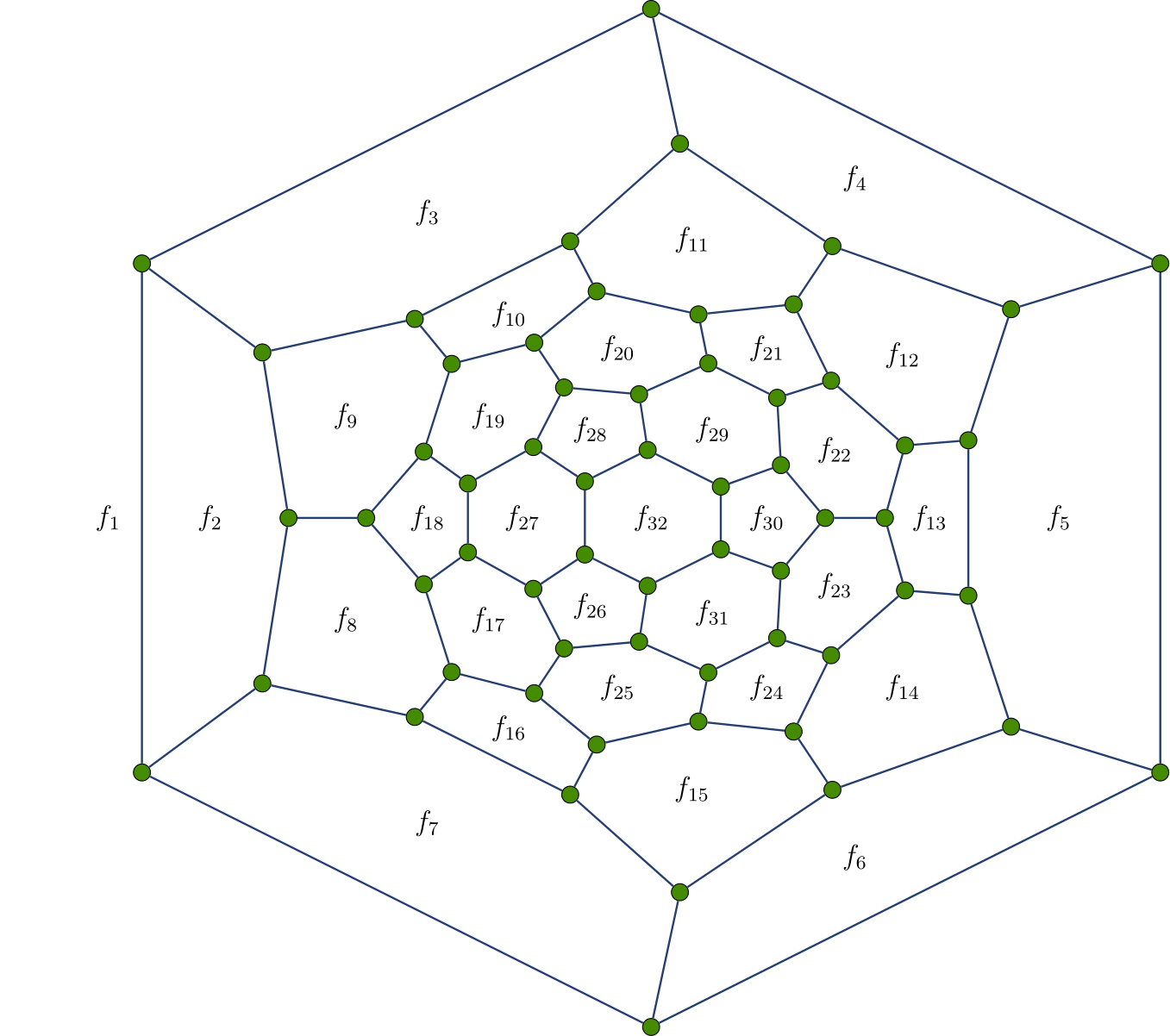}
\label{fig:buckyball_diff}
}
\caption{Two different, but combinatorially equivalent, Schlegel diagrams of Buckminster fullerene }
\label{fig:Schlegel}
\end{figure}

Assume that $C_n$ has a {\it  facet spiral } which is defined as an order of facets such that each facet shares an edge with the previous and next one. This spiral can be presented by a \textit{facet spiral sequence}. It is a sequence of twelve integers, which determines the position of the twelve pentagons in the facet spiral. Another representation is a sequence of fives and sixes such that the $k$--th number in the sequence indicates whether the $k$--th facet in the spiral is a pentagon or a hexagon \cite{AndKarSkr16}. We use the first approach in our software \cite{AB}. For instance, the facet spiral sequence for Buckminster fullerene represented by its Schlegel diagram in Figure  \ref{fig:buckyball} is 
$\left(1,7,9,11,13,15,18,20,22,24,26,32\right)$.  By $C_{n,i}$ we mean the $i^\text{th}$ $C_n$--isomer according to the lexicographical order of facet spiral sequences, see also \cite[Chapter 2]{FowlerManop}.


\subsection{Dual facet graphs, adjacency matrices  and their spectra}
\label{subsec:dual}

Let  $G=(V(G),E(G))=(V,E)$ be a finite undirected graph with vertex set $V$ and edge set $E$. Let $|V|=m$ be the number of vertices in $V$. The adjacency matrix $A_G=A=(a_{i,j})$ of $G$ is given by 
\begin{equation*} 
a_{i,j}:=
\begin{cases}
1, & \text{if}~ (i,j)\in E,\\
0, &\mathrm{otherwise,}\\
\end{cases}, \quad 1\leq i\not= j\leq m, \quad a_{i,i}=0, \quad 1\leq i \leq m.
\end{equation*}
In matrix form, $A$ is a symmetric $m\times m$--matrix with zeros on the diagonal and $\sum_{j=1}^m a_{i,j} $ being equal to the valency of the node $i$:
\begin{equation*}
A=
\begin{pmatrix}
0 && a_{1,2} && \dots && a_{1,m} \\
a_{2,1} && 0 && \dots && a_{2,m}\\
\vdots && \vdots && \ddots && \vdots\\
a_{m,1} && a_{m,2} && \dots && 0
\end{pmatrix}.
\end{equation*}
Define the {\it spectrum} $\sigma(A)$ of $A_G$ as a set of its eigenvalues $\lambda_i(A)=\lambda_i(G)=\lambda_i$, $i=1,\ldots, m$. 
An \textit{induced subgraph} $H$ of $G$ is a graph with vertices set $V(H)\subseteq V(G)$ and all of the edges of $G$ connecting pairs of vertices in $V(H)$. 

Due to the symmetry of $A$, it holds $\sigma(A)\subset \R$.  Let $\text{tr}(A)=\sum_{i=1}^m a_{i,i}$ be the trace of $A$. It obviously holds  $\text{tr}(A)=\sum_{i=1}^m \lambda_i(A)$. Later the \textit{Newton polynomial}
 $N(A,k):=\text{tr}(A^k)=\sum_{i=1}^m \lambda_i^k(A)$ \textit{of degree $k$} with $k\in\N$ and an adjacency matrix $A$, will be of interest to us. It is well--known that the spectrum of an $m\times m$--matrix $A$ can be uniquely restored from the values $N(A,k),$ $k=1,\ldots,m$, cf. e.g. \cite[p. 93]{Gant04}.

 \begin{lemma}\label{lemm:Newton} \label{lemm:NewtonPolynomials}
Let $k\leq m$ be an integer and $A$ be the adjacency matrix of a graph $G$ with $m$ vertices. Then
\begin{enumerate}[a)]
\item \label{Lemma_1_a} the Newton polynomials can be calculated recursively as
 \begin{equation}\label{eq:Newton}
 N(A,k)= -k\sum_{|H|=k}(-1)^{e(H)+c(H)}2^{c(H)}-\sum_{j=2}^{k-2}N(A,k-j) \sum_{H: |H|=j}   (-1)^{e(H)+c(H)} 2^{c(H)},  
 \end{equation}
 where the inner sum runs over all subgraphs $H$ of $G$   with $j$ nodes and connected components being either edges or cycles, $e(H)$ being the numbers of edges among these components and $c(H)$ being the number of cycles. 

\item \label{Lemma_1_b} the Newton polynomial of degree $k$ can be interpreted as the number of all cycles of length $k$ in $G$.
\end{enumerate}
 \end{lemma}
 Here, we call a cycle of length $k$ any closed path (possibly with self--intersections) with $k$ edges from a vertex to itself. 
 \begin{proof}\mbox{ }
 \begin{enumerate}[a)]
\item It is known that Newton polynomials can be represented as polynomials of elementary symmetric polynomials of the eigenvalues $\lambda_i(A)$, $i=1,\ldots,m$ with integer coefficients, see \cite[Chapter 11, \S 53]{Kurosh75}. Since each elementary symmetric polynomial $S_k$ of $\lambda_1(A), \ldots, \lambda_m(A)$ is a sum of principal minors of $A$ of the corresponding degree  $k\in\N$ (cf. \cite[p. 495]{CarlMeyer}), and these minors have integer values due to $a_{i,j}\in \{0,1\}$, we get that the values of $N(A,k)$ are integers. 
 Moreover, $S_1=N(A,1)=0$.    For $k\ge 2$ we have  
\begin{equation}\label{eq:Newtonk<m}
 N(A,k)= -kS_j + \sum_{j=2}^{k-2}(-1)^{j-1}N(A,k-j)S_j,  \end{equation}
 where we put $S_j=0$, $j>m$. By \cite[Theorem 3.10]{Bap}, it holds 
 \begin{equation}\label{eq:SymmetricPol}
 S_j=(-1)^j\sum_{H: |H|=j} (-1)^{e(H)+c(H)} 2^{c(H)}.
  \end{equation}
%
%
\item This interpretation follows immediately from \cite[Proposition 1.3.1]{BroHae}, since the $i^\text{th}$ diagonal entry of the $k^\text{th}$ power of $A$ is the number of walks of length $k$ from vertex $i$ to itself.
\end{enumerate}
 \end{proof}

For fullerenes, vertex adjacency matrices and their spectra are well-studied, see \cite[Section 4.5]{AndKarSkr16}. As mentioned above, we generate dual facet graphs $T_n$ out of the Schlegel diagrams of $C_n$ and consider the spectra of their adjacency matrices. In $T_n$ the original facets become vertices, and the original vertices become facets. The edges of the dual graph show adjacency relations between original facets: two nodes of the dual graph are connected by an edge if the corresponding facets of the fullerene are adjacent, i.e. share an edge. In Figure \ref{fig:bucky_dual}, one can see the dual graph $T_{60}$ of all facets of Buckminster fullerene with the Schlegel diagram in Figure \ref{fig:buckyball}. Notice that (for the sake of legibility) the facet $f_1$ is displayed five times in Figure \ref{fig:bucky_dual}, whereas it should occur just once. In this paper, we use red, green and white nodes in the images of the dual graphs of $C_n$ for pentagons, hexagons and unspecified facets, respectively.

Consider two important induced subgraphs $T^5_{n}$ and $T^6_{n}$ of $T_n$. The graph $T^5_{n}$ illustrates the connectivity between pentagonal facets, i.e. it always contains 12 vertices. For instance, the graph $T^5_{n}$ of every IPR-isomer consists of 12 disconnected vertices. This being said, it is evident that the number of isomers of $C_n$ with the very same graph $T^5_{n}$ increases rapidly with increasing $n$, cf. Figure \ref{fig:number_isomere} for the IPR-case. Hence, considering the graph $T^5_{n}$ does not yield an invariant for all $C_n$-isomers. In order to characterize all isomers we need the graph $T^6_{n}$ showing the connectivity of all $m_6$ hexagonal facets of a $ C_n$-isomer. As an example the graph $T^6_{60}$ of the Buckminster fullerene is shown in Figure \ref{fig:bucky_dual_hexa}. It turns out that $T^6_{n}$ completely characterizes the graph $T_n$. We denote by $A_n, A_n^5,A_n^6$ the adjacency matrix of $T_n, T_n^5$ and $T_n^6$, respectively.

\begin{remark}\label{remark_1}
	\begin{enumerate}[a)]
	\item For $k>m$ a formula similar to (\ref{eq:Newton}) can be derived, by which it follows that traces of the $k^\text{th}$ power of $A$ with $k>m$ are linear combinations of traces of smaller powers. This can be explained by the fact that subgraphs have at most as much vertices as the whole graph.\label{remark:1a}
	\item An alternative approach is to insert formula (\ref{eq:Newtonk<m}) into itself, which yields a representation of Newton polynomials as a polynom with several unknowns being Newton polynomials of lower degree.
		\item The condition that every vertex in a fullerene has valency three corresponds to the fact that the dual graph $T_n$ consists of triangles only. However, the subgraphs $T^5_{n}$ and $T^6_{n}$ may also have larger cycles.  
		\item No 4-cycles exist neither in $T_{n}$ nor in $T^5_{n}$ and $T^6_{n}$, see \cite[Theorem 4.15 (1)]{BuchstEro}.\label{remark_1_b}
\item The problem of description of all  simple cycles  (i.e.,  closed loops without self--intersections) of pentagonal or hexagonal facets of length $k$ is crucial to combinatorial classification of fullerenes.
	\end{enumerate}
\end{remark}
\begin{figure}[H]    								
\centering 	
\subfigure[Dual graph $T_{60}$ of all facets]{									
\includegraphics[scale=0.15]{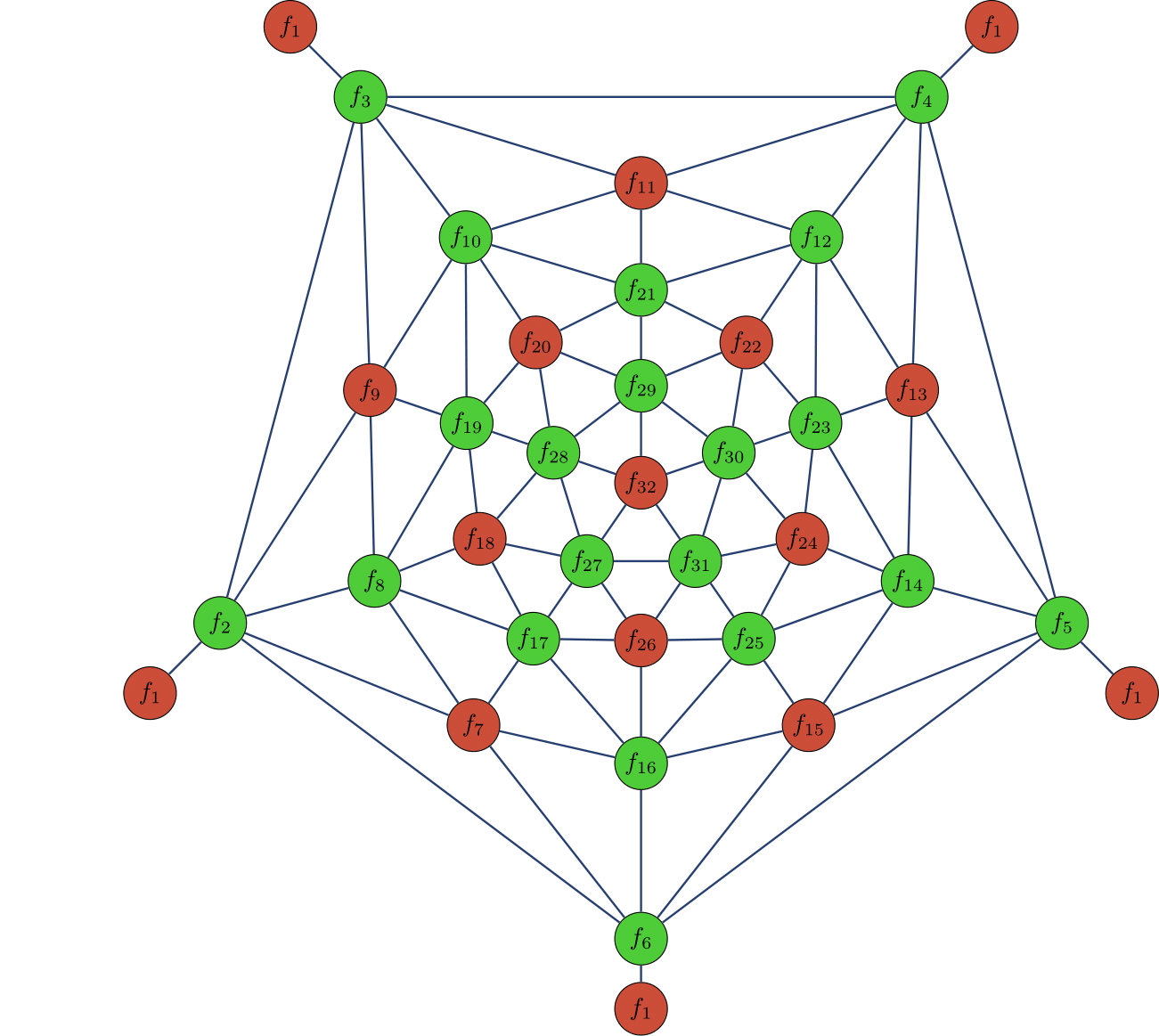}
\label{fig:bucky_dual} 	
}	
\subfigure[Hexagonal dual graph $T^6_{60}$] {
\includegraphics[scale=0.15]{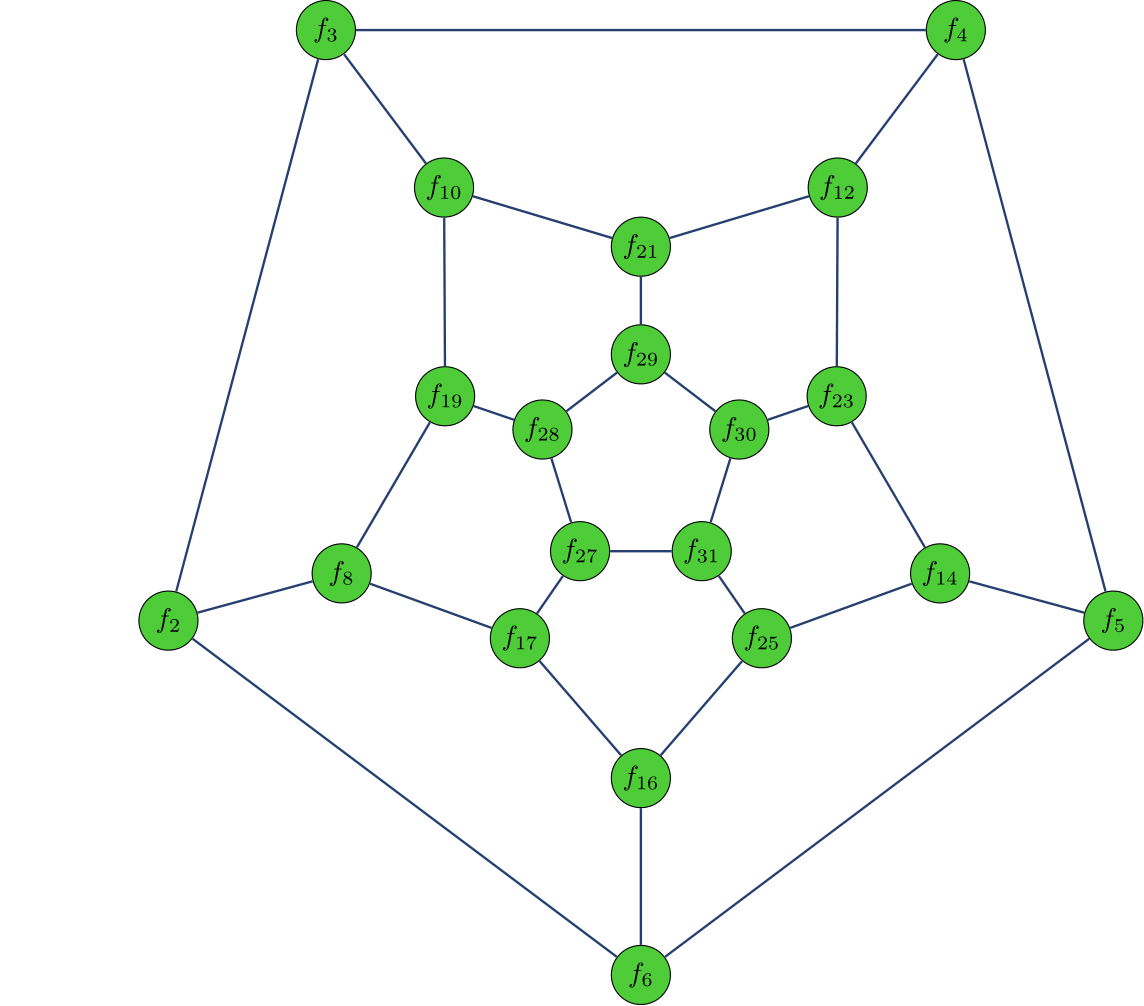}
\label{fig:bucky_dual_hexa}
}
\caption{Dual facet graphs of Buckminster fullerene}
\end{figure}
Denote by $ \N_0$ the set of natural numbers and zero. Obviously, it holds
 $N(A^6_n,k)\in \N_0$ for all $k\in \N$. 

For the dual graphs $T^5_{n}$ and $T^6_{n}$ we construct their adjacency matrices $A_{n}^5$ and $A_{n}^6$.
Since the number of unit entries in each line does not exceed $5$ for $A^5_{n}$ or $6$ for  $A^6_{n}$ it holds by Gershgorin's theorem that
\begin{align*}
\sigma( A_{n}^5) \subset [-5,5],~~ \sigma(A_{n}^6) \subset [-6,6]. 
\end{align*}
In the sequel, let us concentrate on the properties of $\sigma( A^6_{n})$.
The classical Frobenius--Perron theory applied to graph spectra \cite[Proposition 3.1.1]{BroHae}  yields a more accurate estimate for the largest eigenvalue of  $A_{n}^6$ which is positive and of multiplicity one if $T_n^6$ is connected.

\begin{lemma}\label{lemma:bounds_for_kappa}
Let $G$ be a graph with $m$ vertices and valencies $\kappa_1,\ldots,\kappa_m$, and $H$ be an induced subgraph of $G$. Let $\lambda_{\max}(G)$ and $\lambda_{\max}(H)$ be the largest eigenvalues of the adjacency matrix $A_G$ and $A_H$.
\begin{enumerate}[a)]
\item If $G$ is connected, then $\kappa_{\min}\leq\bar{\kappa}\leq \lambda_{\max}(G)\leq \kappa_{max}$, with $\bar{\kappa}$ being its mean valency and $\kappa_{\max}$ the maximum valency of the vertices in $G$. In particular, $\lambda_{\max}(G)=\kappa$ holds if $G$ is $\kappa$-regular. 
\item It holds $\sqrt{\frac{1}{m}\sum_{i=1}^m \kappa_i^2}\leq \lambda_{\max}(G)\leq \kappa_{\max}$.
\item It holds $\lambda_{\max}(H)\leq \lambda_{\max}(G)$.
\end{enumerate}

\begin{proof}
\begin{enumerate}[a)]
\item See \cite[Proposition 3.1.2]{BroHae}.
\item See \cite[Theorem 1.2]{largest_eigenvalue_of_a_agraph} and \cite[Comment on Proposition 3.1.2]{BroHae}.
\item See \cite[Lemma 3.16]{Bap}.
\end{enumerate}
\end{proof}

\end{lemma}
Notice that for all connected irregular graphs $T_n,T_n^5$ and $T_n^6$ we have $\bar{\kappa}>1$ and $\kappa_{max}\le 6$. For non--connected or irregular graphs we get only $0\le \bar{\kappa} \le \kappa_{max}$.
For instance, the dual graph $T^6_{60}$  of the Buckminster fullerene (see Fig. \ref{fig:bucky_dual_hexa})  is connected and regular  with $\kappa=3$, thus the  largest eigenvalue of $A_{60,6}$ for it is equal to 3.

In general, the value $\theta:=\kappa_{max} - \bar{\kappa}\ge 0$ is a measure of the asymmetry of the dual facet graph which we call  the {\it asymmetry  coefficient}.

\begin{remark}
If $k\to\infty$, we have $N(A_n^6,k)/ \lambda_{max}^{k} \sim a_{max}+(-1)^k {\bf 1}(- \lambda_{max}\in\sigma(A_{n}^6))$, where ${\bf 1}(B)$ is the indicator function of $B$ and $a_{max}$ is the multiplicity of the eigenvalue $\lambda_{max}$. Indeed, by \cite[Theorem 6.3]{Bap}, it holds $- \lambda_{max}\in\sigma( A_{n}^6)$ iff our dual graph $T_{n}^6$  is bipartite, i.e., it has no cycles of all odd lengths, cf. \cite[Corollary 3.12]{Bap}.    In this case,  $- \lambda_{max}$ has necessarily multiplicity one, see  \cite[Lemma 3.13]{Bap}. Since for large $n\geq 40$ the dual graph $T_{n}^6$ of hexagons of any isomer of $C_n$ contains either a 3-- or a 5--cycle, it holds  $- \lambda_{max}\not\in\sigma(A_n^6)$, and
  the behavior of the the whole Newton polynomial $\text{tr}\left(\left(A_{n}^6\right)^{k}\right)$  for large powers $k$ is dominated by $ \lambda_{max}^{k}$.
\end{remark}

Let $|A|$ be the cardinality of a finite set $A$.
\begin{lemma}\label{lemma:edges_T6-T5}
For all feasible $n$, consider a $C_n$--isomer with graphs $T_{n}$, $T_{n}^5$ and $T_{n}^6$. Then it holds
\begin{equation*}
|E(T^6_{n})|=|E(T^5_{n})|+ \frac{3n}{2}-60.  
\end{equation*} 
\end{lemma}
\begin{proof}
Since the graph $T_n$ of a $C_n$--isomer is  $3$--regular, it has $\frac{3n}{2}$ edges. An edge $(v,w)\in E(T_n)$ is either an edge between two pentagons, $(v,w)\in E(T_n^5)$, or between two hexagons, $(v,w)\in E(T_n^6)$, or between a pentagon and a hexagon, $(v,w)\in E(T_n\setminus \left(T_n^5 \cup T_n^6\right))$. Since the number of pentagons is always $12$, there are $60-2|E(T_n^5)|$ edges between pentagons and  thus 
$$
|E(T_n^5)|+|E(T_n^6)|+60-2|E(T_n^5)|=\frac{3n}{2},
$$
which finishes the proof.  
\end{proof}

\section{Cospectral Isomers}\label{section:cospectral_isomers}

\begin{definition}
Let $A_G$,$A_H$ be adjacency matrices of graphs $G$ and $H$ with $m$ vertices each. $G$ and $H$ (or $A_G$ and $A_H$) are said to be \textit{cospectral} if $\sigma(A_G)=\sigma(A_H)$.
\end{definition}

It is easy to prove that isomorphic graphs are cospectral. The inverse statement is in general not true (cf. e.g. a counterexample in \cite{which_graphs_are_determined_by_their_spectrum}). The natural question arises: {\it Which graphs are determined by their spectrum?} \cite{which_graphs_are_determined_by_their_spectrum}. Some specific graphs like paths, complete graphs, regular complete bipartite graphs, cycles and their complements yield a positive answer to this question. In what follows, we provide new examples of non--isomorphic cospectral graphs coming from the world of fullerenes.

We derive some theoretical results for sets of non-cospectral graphs and apply them to fullerene isomers. To begin with,  recall the following

\begin{lemma}\label{lemma:cospectral_equivalency}
For graphs $G$ and $H$ with $m$ vertices and adjacency matrices $A_G$ and $A_H$ the following statements are equivalent:
\begin{enumerate}[a)]
\item $G$ and $H$ are cospectral.\label{lemma:cospectral_equivalency_a}
\item $A_G$ and $A_H$ have the same characteristic polynomial.\label{lemma:cospectral_equivalency_b}
\item $N(A_G,k)=N(A_H,k)$ for all $k=1,\ldots,m$.\label{lemma:cospectral_equivalency_c}
\end{enumerate}
\end{lemma}
\begin{proof}
Equivalence of \ref{lemma:cospectral_equivalency_a}) and \ref{lemma:cospectral_equivalency_b}) is obvious.

Now prove the equvalence of \ref{lemma:cospectral_equivalency_c}) and \ref{lemma:cospectral_equivalency_b}). As stated in  \cite[Chapter 11, \S 53]{Kurosh75} elementary symmetric polynomials $S_k$ of degree $k$ of eigenvalues of $A$ can be expressed as   
\begin{align*}
S_k=(-1)^{k-1}\frac{N(A,k)}{k}-\frac{1}{k}\sum_{i=1}^{k-1}(-1)^i S_{k-i}N(A,i), ~~~ 1\leq k\leq n.
\end{align*} 
Since $S_1=N(A,0)=0$ and $S_2=-N(A,2)/2$, every $S_k$ with $2\leq k\leq n$ can be computed just knowing $N(A,2),\ldots,N(A,k)$. Hence, the symmetric polynomials are identical for both $A_G$ and $A_H$. Finally, the coefficients of the characteristic polynomials can be expressed as $(-1)^kS_k$, so the characteristic polynomials are equal as well. 
\end{proof}

 For all even $24\leq n \leq 150$, we checked the existence of cospectral pairs of  $C_n$-isomers within $T_n$, $T^5_{n}$ and $T^6_{n}$ graphs. Using Lemma \ref{lemma:cospectral_equivalency}, we  applied  MATLAB functions {\tt eig()}, {\tt charpoly()} and {\tt trace()} with double precision, cf. \cite{Matlab}. In cases where two isomers seem to be cospectral with respect to $T_n$ or $T_n^6$, we increase the precision by using {\tt vpi} format and MATLAB symbolic toolbox \cite{Matlab}.

For $n < 32$ no pair of cospectral isomers with respect to $T_n,T_n^5$ and $T_n^6$ can be found. Our results for $32\leq n \leq 60$ are listed in Table \ref{tab:cospectral}. This table can be extended to $n>60$ with all zeroes in its first and third rows, and positive integers in its second row. 

Considering the whole dual graph $T_n$ one finds only one pair of cospectral isomers with $n=44$, compare Figures \ref{fig:C_44_37} and \ref{fig:C_44_38}. To explain the difference within this pair, we need the following 

\begin{definition}
Let $G$ and $H$ be two graphs.
\begin{enumerate}[a)]
\item A fragment $F$ of $G$ is a connected induced subgraph of $G$. In particular, for fullerenes we call fragments of $T_n^5$ and $T_n^6$ pentagon-fragments and hexagon-fragments of $T_n$, respectively.
\item Two non-isomorphic graphs $G$ and $H$ are called fragment flipped if two isomorphic fragments $F_G$ in $G$ and $F_H$ in $H$ exist such that the remaining graphs $G\setminus F$ and $H\setminus F$ are isomorphic. 
\end{enumerate} 
\end{definition} 

As one can see in Figure \ref{fig:C_44_37} and \ref{fig:C_44_38}, the two cospectral $C_{44}$-isomers are fragment flipped. However, in general such a flip does not preserve the spectrum of a graph. To illustrate this, Figure \ref{fig:C_60_4} and \ref{fig:C_60_5} contains a non--cospectral fragment flipped pair of $C_{60}$--isomers.

For $n\in\lbrace 32, 36,40, 52\rbrace$ a pair of distinct isomers exists with the same spectrum $\sigma\left(T_{n}^6\right)$, since their graphs $T_{n}^6$ are isomorphic. 

\begin{table}[H]    								
	\centering
	\begin{tabular}{|r || c | c | c | c | c | c | c | c | c | c | c | c | c | c | c |}\hline
		$n$ & 32 & 34 & 36 & 38 & 40 & 42 & 44 & 46 & 48 & 50 & 52 & 54 & 56 & 58 & 60 \\ \hline
		 $T_n$ & 0 & 0 & 0 & 0 & 0 & 0 & 1 & 0 & 0 & 0 & 0 & 0 & 0 & 0 & 0  \\
		$T^5_{n}$ & 0 & 0 & 0 & 0 & 3 & 3 & 15 & 16 & 42 & 63 & 95 & 112 & 148 & 147 & 177 \\
		 $T^6_{n}$ & 1 & 0 & 1 & 0 & 1 & 0 & 0 & 0 & 0 & 0 & 1 & 0 & 0 & 0 & 0 \\
		 \hline
	\end{tabular}
	\caption{Number of non-unique spectra of graphs $T_n$, $T^5_{n}$ and $T^6_{n}$ of $C_{n}$--isomers.}
	\label{tab:cospectral}
\end{table}

For a fixed $40\leq n\leq 150$ at least three and at most 298 non-unique spectra $\sigma(T_n^5)$ exist. Since the amount of different arrangements of 12 pentagons is limited and the number of isomers grows rapidly with increasing $n$, there must exist isomers with same subgraphs $T_n^5$. We discuss the number and different configurations of pentagon-fragments in \cite{BBS_2} in more detail.
 
The above empirical findings lead to the following

\begin{conjecture}\label{con:cospectral}
\begin{enumerate}[a)]
\item For all feasible $n \not=44$ two $C_n$--isomers are isomorphic iff they are cospectral with respect to $T_n$.
\item For all feasible $n\geq 54$ two $C_n$--isomers are isomorphic iff they are cospectral with respect to $T_n^6$. \label{con:spectral_b}
\item For any feasible $n$ at least one of the spectra $\sigma(T_n)$, $\sigma(T_n^5)$, $\sigma(T_n^6)$ is unique for all $C_n$--isomers.
\end{enumerate}
\end{conjecture}
 
For a graph $G$ with $m$ vertices we denote by $|\sigma(G)|$ the set of absolute values of eigenvalues $\lambda_i\in\sigma(G)$ for $i=1,\ldots,m$ sorted in descending order, i.e. $|\sigma(G)|:=\lbrace |\lambda_i|\mid \lambda_i \in\sigma(G),~ |\lambda_i|\geq |\lambda_{i+1}|~\forall i=1,\ldots,m-1 \rbrace$, and call it the \textit{absolute spectrum} of $G$. Denote by $|\sigma_{[0,1)}(G)|$ and $|\sigma_{[1,\infty)}(G)|$ the part of $|\sigma(G)|$ with $0\leq|\lambda|<1$ and $1\leq |\lambda|$. It holds $\lambda_1(G)=\lambda_{\max}(G)$
 
\begin{theorem}\label{theorem:main_general}
Let $\Gamma$ be a set of graphs with $m$ vertices and with distinct absolute spectra such that $\lambda_{\max}(G)>1$ for all $G\in\Gamma$. 
Then all graphs $G\in\Gamma$ can be uniquely characterized by at most two Newton polynomials of even degrees $k_1^*$, $k_2^*$ with $m\geq k_1^*\geq k_2^*$.
\end{theorem} 

\begin{proof}
Ordering all absolute spectra $|\sigma(G)|$ lexicographically yields a unique order on the set $\Gamma$. Then all graphs $G$ from $\Gamma$ can be distinguished either by $|\sigma_{[1,\infty)}(G)|$ or by $|\sigma_{[0,1)}(G)|$.

In the first case, let us consider all $G\in \Gamma$ with distinct $|\sigma_{[1,\infty)}(G)|$. The sum $\sum_{\lambda\in |\sigma_{[0,1)}(G)|}|\lambda|^k$ converges to 0 for $k\rightarrow \infty$, so its influence on the Newton polynomials $N(A_G,k)$ for $k$ large enough can be neglected. Since $|\sigma_{[1,\infty)}(G)|$ is unique for all considered graphs $G$ there must exist a degree $k_1^*$ such that the values of $N(A_G,k)$ for all even $k\geq k_1^*$ are distinct. 

Now if some graphs $G\in\Gamma$ have the identical part of absolute spectrum $|\sigma_{[1,\infty)}(G)|$, they must differ within the part   $|\sigma_{[0,1)}(G)|$. Hence,  there must always exist a degree $k_2^*\leq k_1^*$ such that the sum $\sum_{\lambda\in|\sigma_{[0,1)}(G)|}|\lambda|^{k_2^*}$ distinguishes  these graphs.  

Lemma \ref{lemma:cospectral_equivalency} \ref{lemma:cospectral_equivalency_c}) and Remark \ref{remark_1} \ref{remark:1a}) yield the number of vertices in the graph $G$ as an upper bound for $k_1^*$ and $k_2^*$.
\end{proof}

\begin{corollary}
Assuming Conjecture \ref{con:cospectral} to be true, all $C_n$-isomers can be uniquely characterized by at most two Newton polynomials $N(A_G,k_1^*)$ and $N(A_G,k_2^*)$, $k_1^*\leq k_2^*$ even, with respect to at least one of the graphs $G\in\lbrace T_n, T_n^5,T_n^6\rbrace$.
\end{corollary}
\begin{proof}
This follows from Theorem \ref{theorem:main_general} and Lemma \ref{lemma:bounds_for_kappa}, since $\lambda_{\max}(G)>1$ for at least one of the graphs $G\in\lbrace T_n T_n^5,T_n^6\rbrace$.
\end{proof}

\begin{remark}
\begin{enumerate}[(i)]

\item Note that the set of $C_n$--isomers with distinct largest eigenvalues $\lambda_{max}(G)$ for some $G\in\lbrace T_n T_n^5,T_n^6\rbrace$ has distinct sets $|\sigma_{[1,6]}(G)|$.

\item It is important to stress that we consider even degrees $k\le m$ only. In general, values of Newton polynomials $N(A_G,k)$ with odd degrees $k$  do not distinguish between graphs $G$.  To illustrate this point, consider the hexagon graphs of two specific isomers of $C_{28}$. The graph $T_{28,1}^6$ of the first isomer consists of four isolated hexagons and the other graph $T_{28,2}^6$ of two pairs of two adjacent hexagons. One gets the following spectra:
\begin{align*}
\sigma(T_{28,1}^6)=\lbrace 0, 0 ,0,0\rbrace, ~~\sigma(T_{28,2}^6)=\lbrace -1,-1,1,1\rbrace.
\end{align*}
It follows directly that Newton polynomials of any odd degree do not distinguish between $C_{28,1}$ and $C_{28,2}$, but  $N(A_{28,1}^6,k)\not=N(A_{28,2}^6,k)$ for every even $k\geq 2$.
\end{enumerate}

\end{remark}

\begin{figure}[H]    								
\centering 	
\subfigure[Dual graph $T_{44}$ of $C_{44,37}$]{									
\includegraphics[scale=0.15]{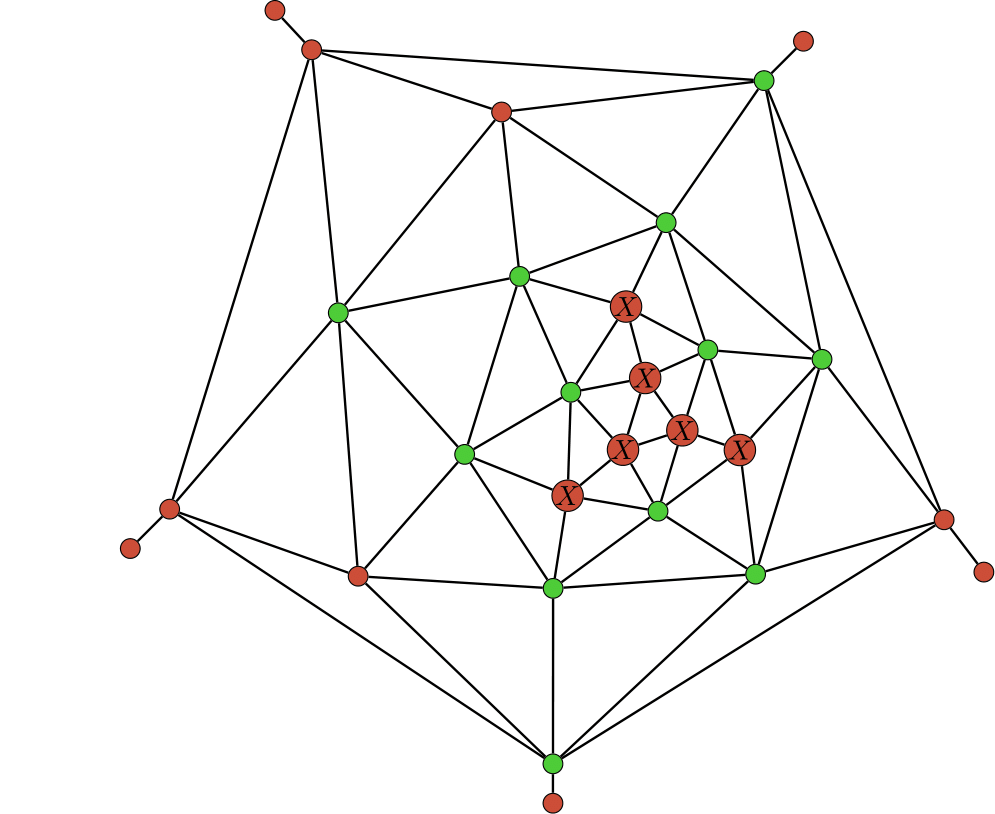}
\label{fig:C_44_37} 	
}	
\subfigure[Dual graph $T_{44}$ of $C_{44,38}$] {
\includegraphics[scale=0.15]{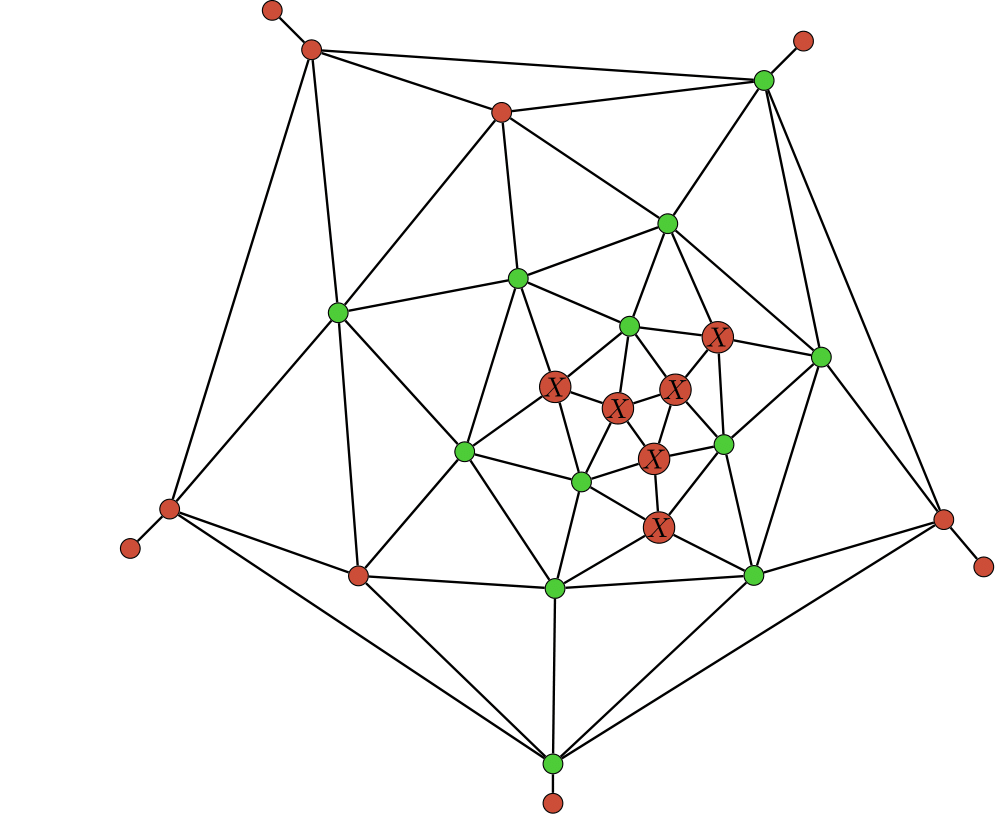}
\label{fig:C_44_38}
}
\\
\subfigure[Dual graph $T_{60}$ of $C_{60,4}$]{									
\includegraphics[scale=0.15]{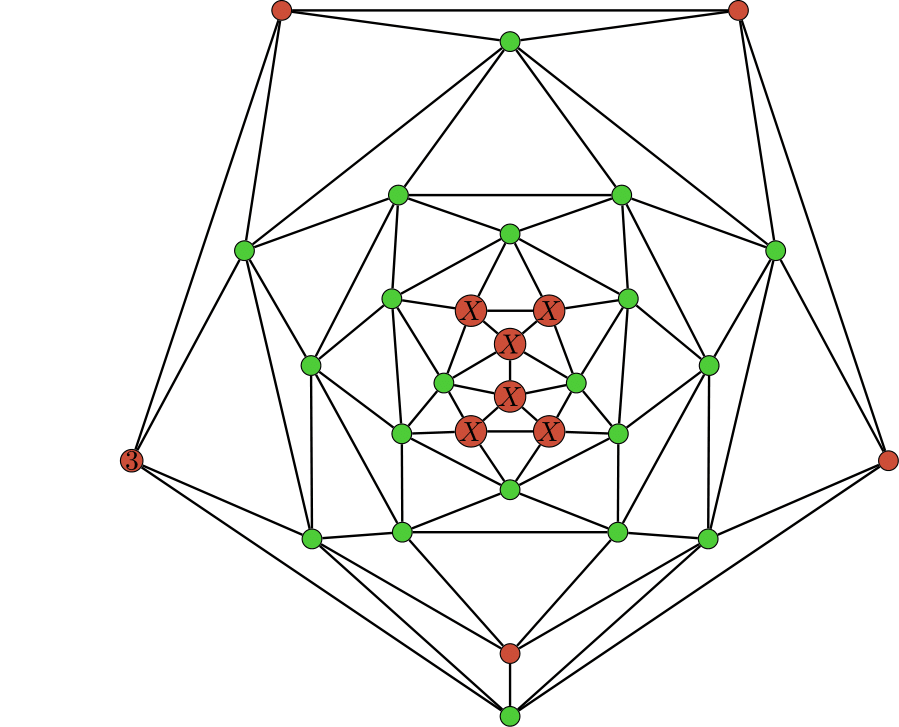}
\label{fig:C_60_4} 	
}	
\subfigure[Dual graph $T_{60}$ of $C_{60,5}$] {
\includegraphics[scale=0.15]{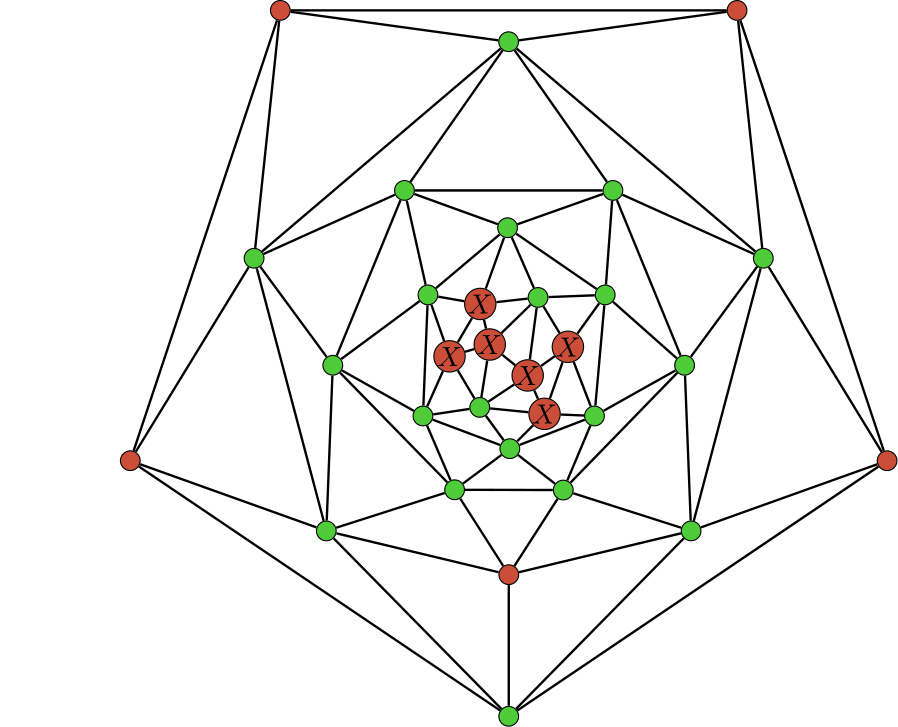}
\label{fig:C_60_5}
}
\caption{Two examples of pentagon-fragment (marked with an X) flipped isomers. $C_{44,37}$ and $C_{44,38}$ are cospectral, but $C_{60,4}$ and $C_{60,5}$ are not.}\label{fig:C44}
\end{figure}

\section{Spectral classification and stability prediction}
\label{sec:spectral_classification_and_stability_prediction}
In this section,  we decompose the family of all $C_n$--isomers into subsets, which we call \textit{clusters}, using the Newton polynomials $N(A_G,k)$, $k=2,4,6,\ldots,k^*$ with adjacency matrix $A_{G}$. This approach can be applied to any graph $G\in\lbrace T_n,T_n^5,T_n^6\rbrace$  with no cospectral isomers.

\begin{definition}\label{def:cluster}
\begin{enumerate}[a)]
\item For a feasible $n$ and a given even integer $k$, we call a set of isomers with same value $N(A_G,k)$  of Newton polynomial of degree $k$ a cluster of $C_n$. For a fixed $n$,  the family of all these clusters is called a clusterization of $C_n$. A clusterization such that its every cluster has exactly one element is called a complete clusterization. 
\item We define $k_{single}^*$ as the minimal degree $k$ of Newton polynomials which is needed for a complete clusterization.
\end{enumerate}
\end{definition}

As we have seen in Lemma \ref{lemm:Newton} \ref{Lemma_1_b}), the Newton polynomial of degree two and three is equal to twice the number of edges and six times the number of triangles in the considered graph. The interpretation of $N(A_n,4)$ is a bit more complex. We have to count all possible cycles of length four. 
 In Figure \ref{fig:cycles_4}, the idea of calculation of $N(A_{60},4)$ is illustrated on  Buckminster fullerene $C_{60, 1812}$.  There, all five possible cycles of length four together with their frequencies are listed for one pentagon in $C_{60, 1812}$.

\begin{figure}[H]
\subfigure[A fragment from Buckminster fullerene]{
\includegraphics[width=0.16\textwidth]{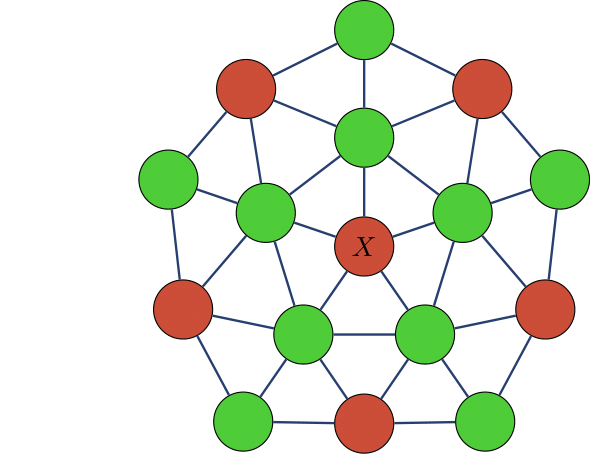}\label{fig:frag}
}\hfill
\subfigure[Appears five times in (a)]{
\hspace*{0.5cm}\includegraphics[width=0.06\textwidth]{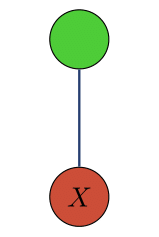}\hspace*{0.5cm}
}\hfill
\subfigure[Appears 25 times in (a)]{
\hspace*{0.5cm}\includegraphics[width=0.05\textwidth]{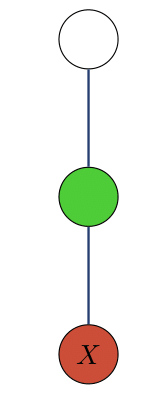}\hspace*{0.5cm}
}\hfill
\subfigure[Appears 20 times in (a)]{
\hspace*{0.5cm}\includegraphics[width=0.05\textwidth]{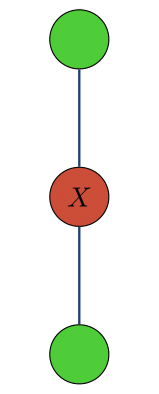}\hspace*{0.5cm}
}\hfill
\subfigure[Appears ten times in (a)]{
\hspace*{0.5cm}\includegraphics[width=0.1\textwidth]{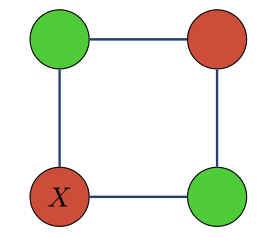}\hspace*{0.5cm}
}\hfill
\subfigure[Appears ten times in (a)]{
\hspace*{0.5cm}\includegraphics[width=0.1\textwidth]{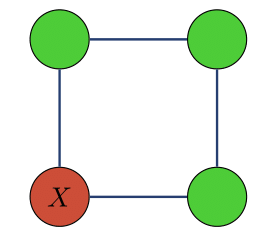}\hspace*{0.5cm}
}
\caption{Five ((b)-(f)) different kinds of cycles of length four beginning and starting in the central vertex (X) in (a)}\label{fig:cycles_4}
\end{figure}   

For the graph $T_n$,  one can show that the sum of these frequencies over all vertices in $T_{n}$  only depends on $n$, cf. \cite{BBS_2}. So, $N(A_n,k)$ for $k\leq 4$ can be neglected for the clusterization of $C_n$ on the basis of $T_n$, since these Newton polynomials have the same value for all $C_n$-isomers.

Nevertheless, for $T_n^6$ we have to consider every Newton polynomial of even degree, since in $T_n^6$ the number of vertices is fixed, but neither the number of edges between them nor the number of triangles in $T_n^6$ is determined by $n$.

In the following section we use $T_{60}^6$ in order to get a complete clusterization of $C_{60}$ as an example.

\subsection{Clusterization of $\boldsymbol{C_{60}}$ using Newton polynomials}
Recall that by Euler's formula each $C_{60}$-isomer has $90$ edges and $32$ facets with $12$  pentagons and $20$ hexagons among them. 

The Newton polynomial $N(A^6_{60},2)$ can take on 18 distinct values $60= t_1 < \ldots < t_{18}=100$. For values $t_1=60, t_2=64, t_{16}=92, t_{17}=96$ and $t_{18}=100$, there exists exactly one $C_{60}$-isomer with $N(A^6_{60},2)=t_i$, $i\in\lbrace 1,2, 16,17,18 \rbrace$. These isomers are $C_{60,1812}, C_{60,1809}, C_{60,2}, C_{60,3},C_{60,1}$, respectively. Their Schlegel diagrams and dual hexagonal graphs $T_{60}^6$ are shown in Figure \ref{fig:schlegel+dual}. Moreover, our numerical results show that for any degree $k\geq 2$ with respect to $T_{60}^6$ these isomers form a cluster with one single element. Ordering these five isomers according to $N(A^6_{60},k)$ does not change with increasing degree $k\geq 2$. In addition, Newton polynomials of all other $C_{60}$-isomers are bounded by  $N(A_{60,1809}^6,k)$ and $N(A_{60,2}^6,k)$, i.e. $N(A_{60,i}^6,k)\in\left( N(A_{60,1809}^6,k), N(A_{60,2}^6,k) \right)$ for all $i\notin\lbrace 1,2,3,1809,1812 \rbrace$ and all $k\geq 2$. 

\begin{figure}[H]
\subfigure[$C_{60,1}$]{
\includegraphics[width=0.17\textwidth]{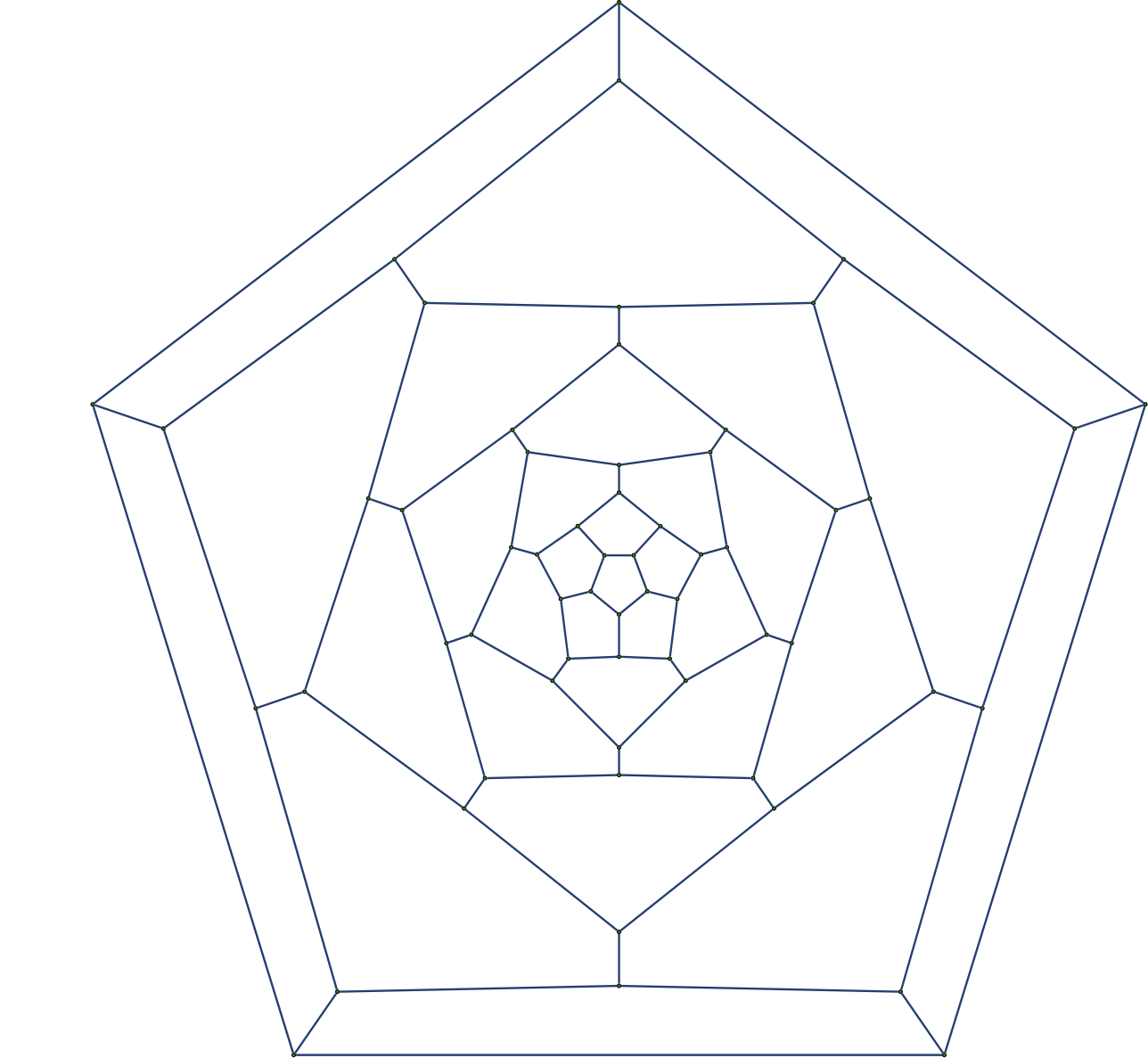}
\label{fig:C_60,1}
}
\subfigure[$C_{60,2}$]{
\includegraphics[width=0.17\textwidth]{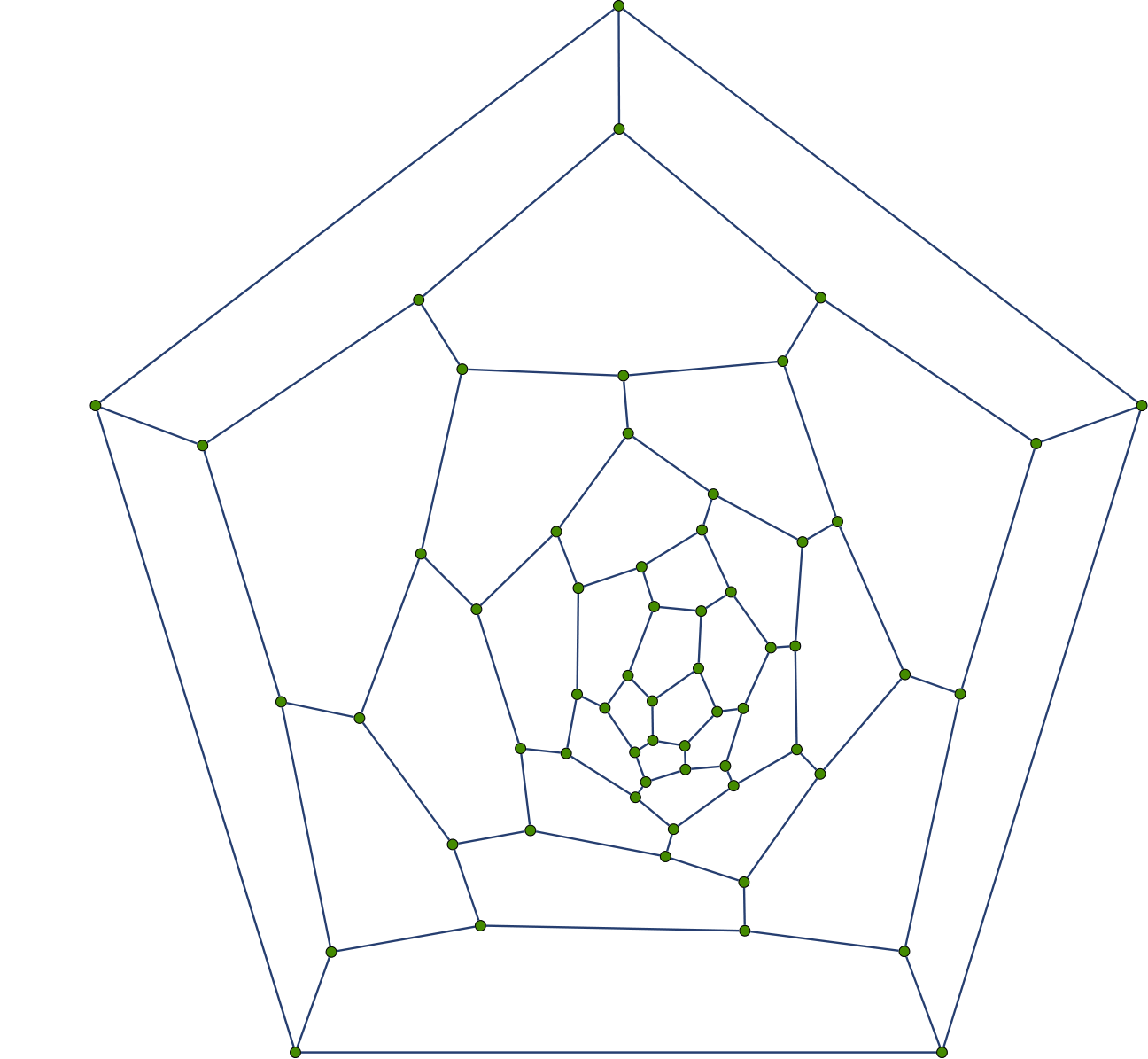}
\label{fig:C_60,2}
}
\subfigure[$C_{60,3}$]{
\includegraphics[width=0.17\textwidth]{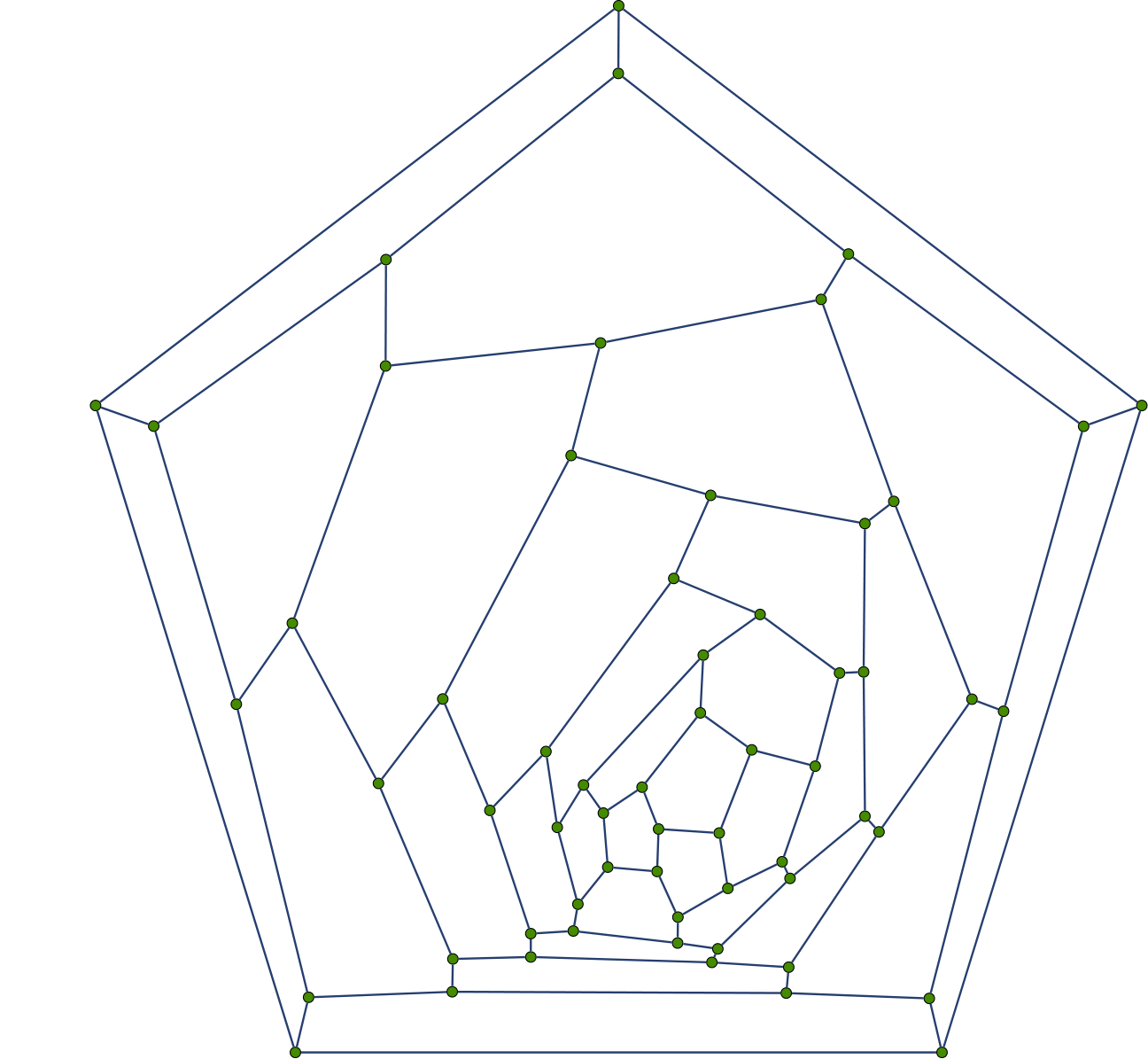}
\label{fig:C_60,3}
}
\subfigure[$C_{60,1809}$]{
\includegraphics[width=0.17\textwidth]{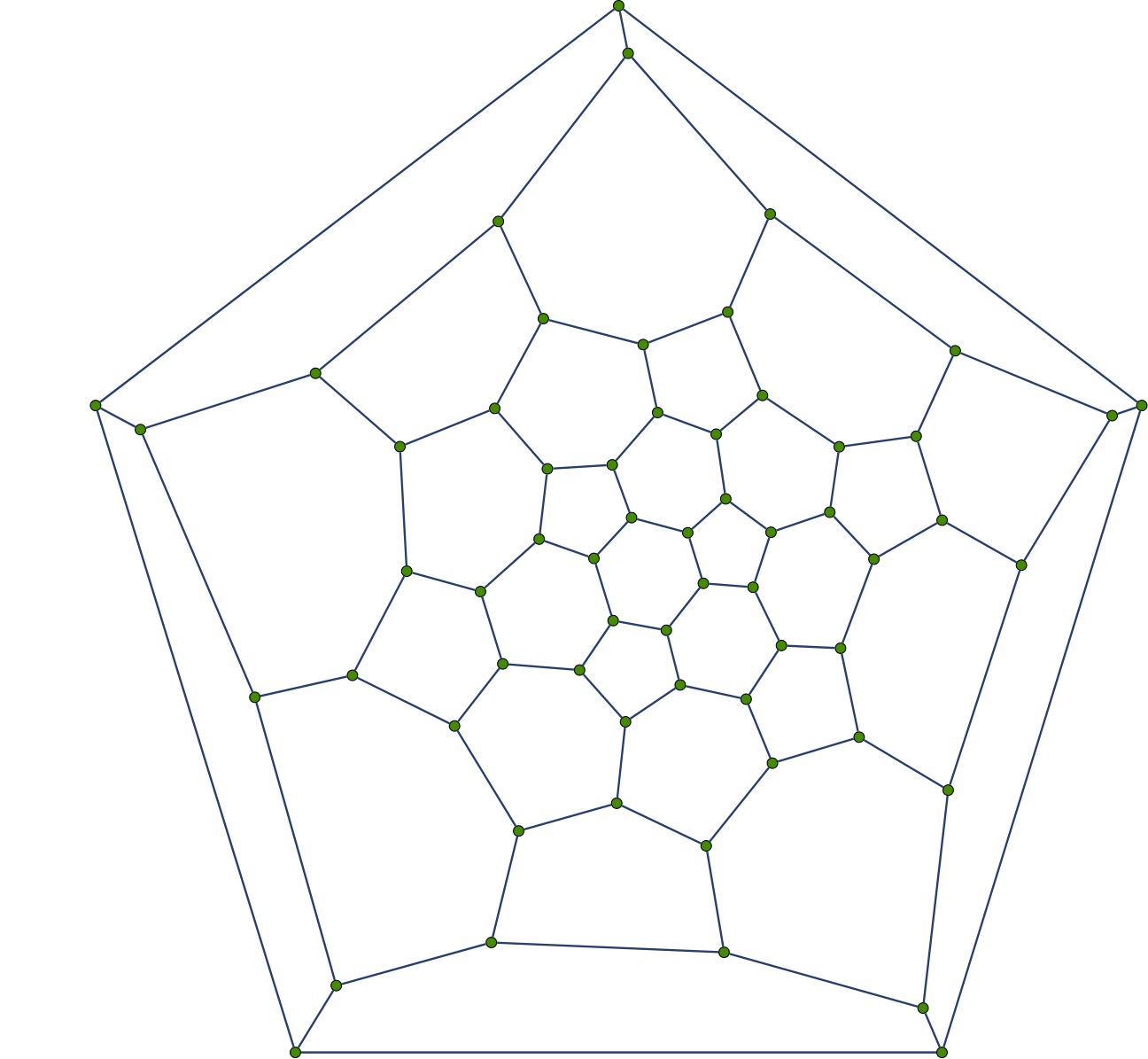}
\label{fig:C_60,1809}
}
\subfigure[$C_{60,1812}$]{
\includegraphics[width=0.17\textwidth]{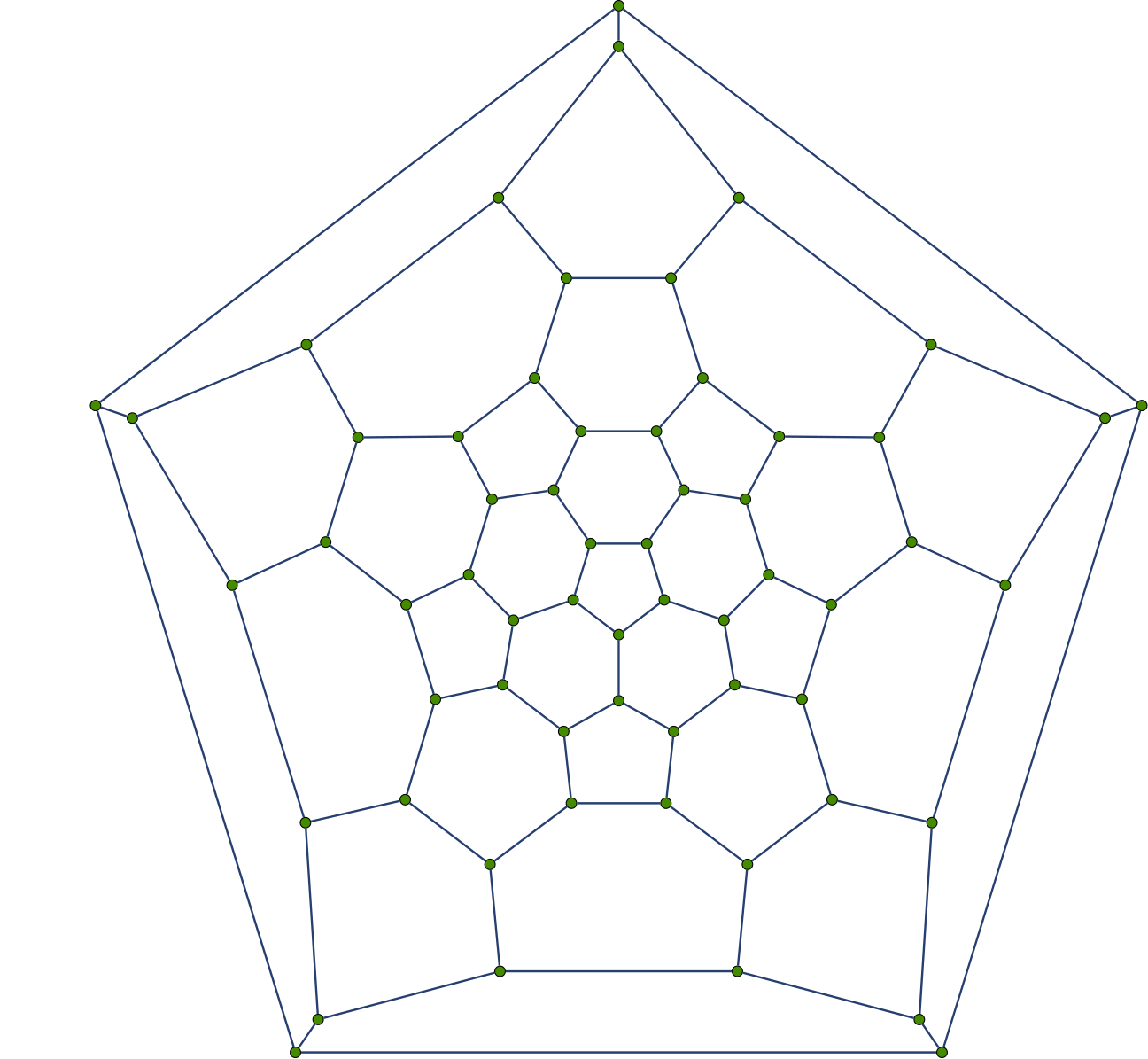}
\label{fig:C_60,1812}
}
\subfigure[$T_{60,1}^6$]{
\includegraphics[width=0.17\textwidth]{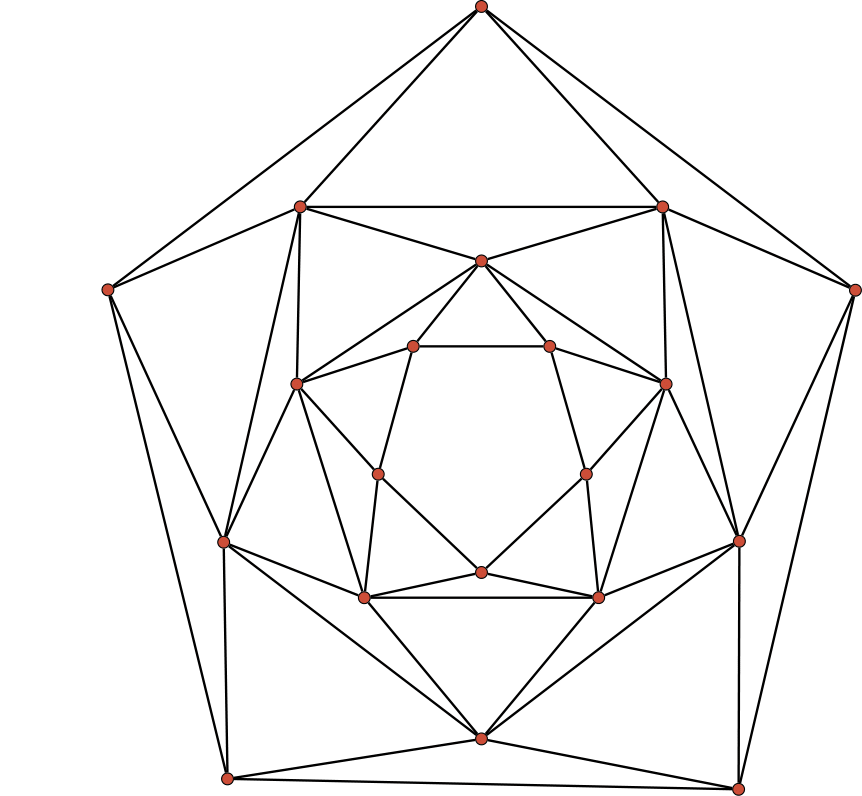}
\label{fig:Dual_Hexa C_60,1}
}
\hspace*{0.1cm}
\subfigure[$T_{60,2}^6$]{
\includegraphics[width=0.17\textwidth]{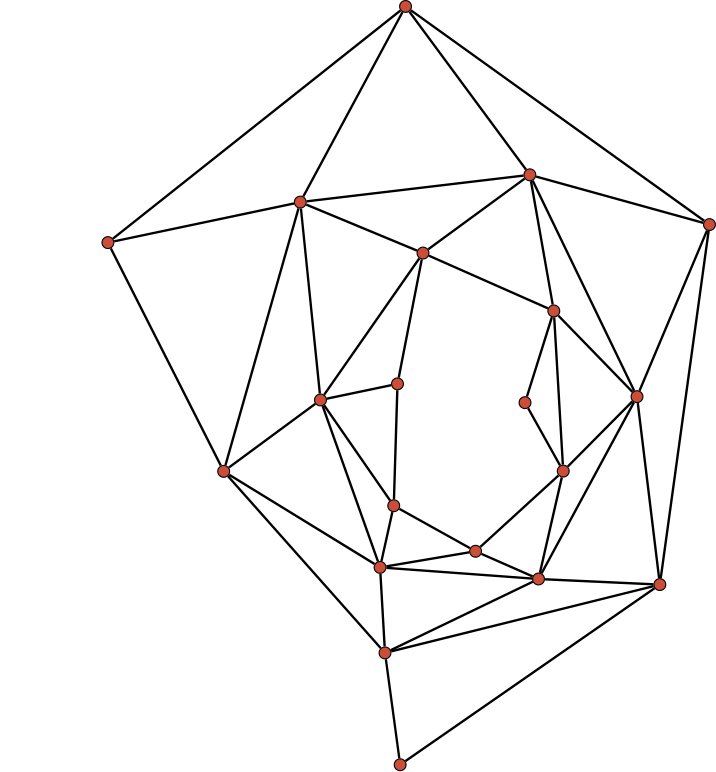}
\label{fig:Dual_Hexa C_60,2}
}
\hspace*{0.1cm}
\subfigure[$T_{60,3}^6$]{
\includegraphics[width=0.17\textwidth]{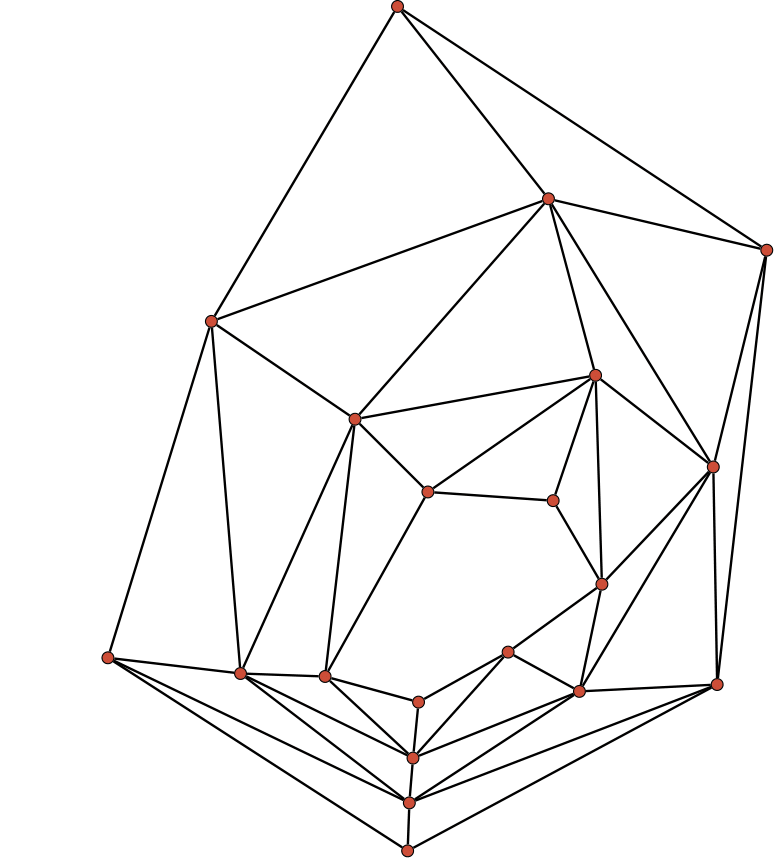}
\label{fig:Dual_Hexa C_60,3}
}
\hspace*{0.1cm}
\subfigure[$T_{60,1809}^6$]{
\includegraphics[width=0.17\textwidth]{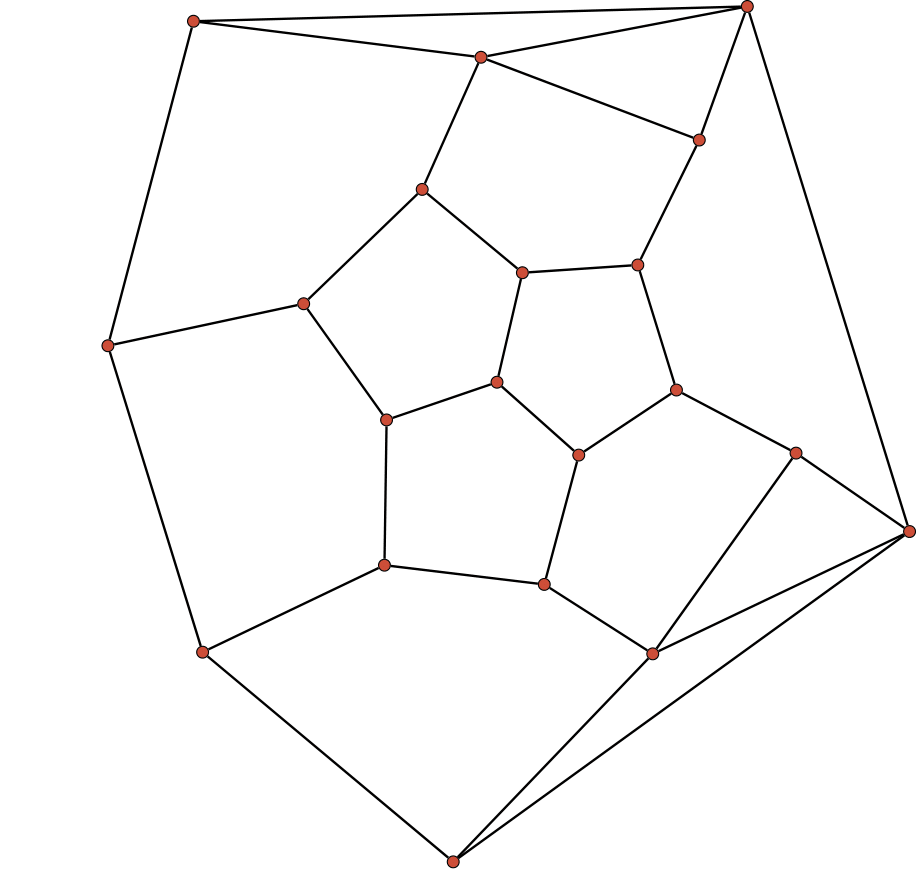}
\label{fig:Dual_Hexa C_60,1809}
}
\hspace*{0.1cm}
\subfigure[$T_{60,1812}^6$]{
\includegraphics[width=0.17\textwidth]{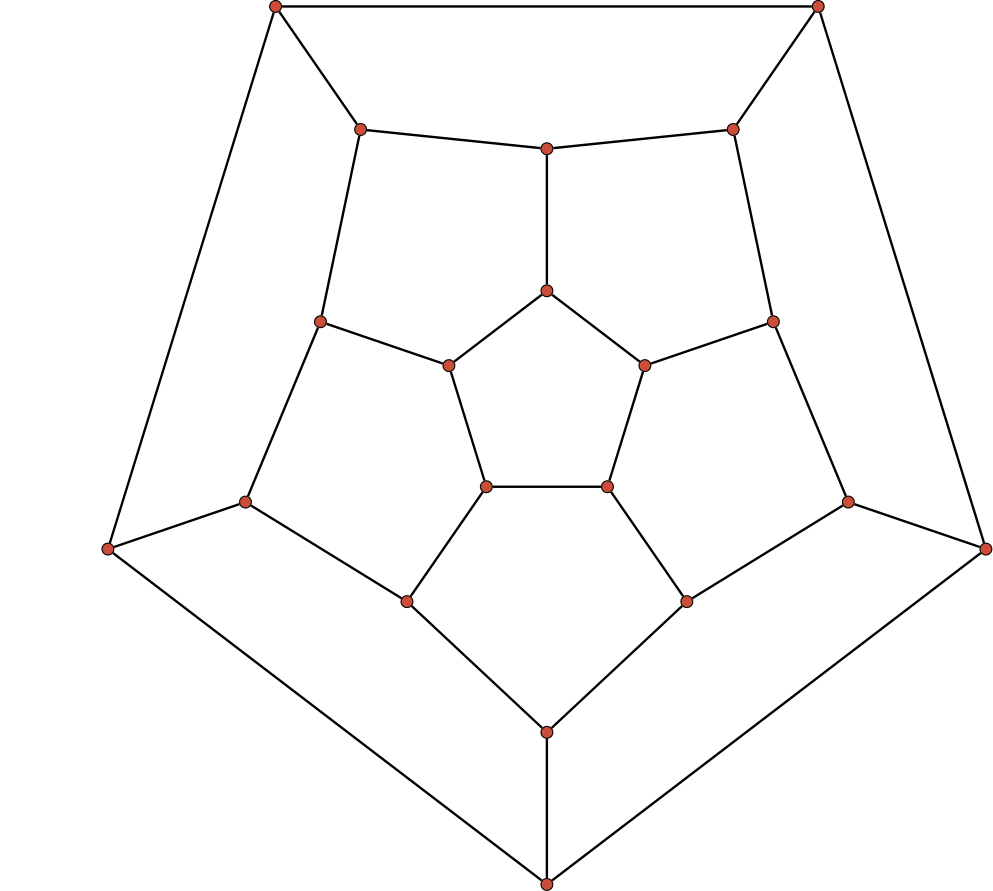}
\label{fig:Dual_Hexa C_60,1812}
}
\hspace*{0.1cm}
\caption{Schlegel diagrams and dual graphs of hexagons of five $C_{60}$-isomers}
\label{fig:schlegel+dual}
\end{figure}

We checked that for any pair of two $C_{60}$-isomers with distinct Newton polynomials of degree $k$ the Newton polynomials $N(A_{60}^6,\tilde{k})$ with $k\leq \tilde{k}\leq 100$ are distinct as well. So, it holds $k_1^*=k_2^*=k_{single}^*=12$, where $k_1^*$ and $k_2^*$ are from Theorem \ref{theorem:main_general}. One can observe that the number of distinct Newton polynomials and so the number of clusters with one element is monotone growing with $k$. Numbers of clusters and clusters with one element for all even $2\leq k\leq 100$ are listed in Table \ref{tab:Hist NewPoly}.

\begin{table}[H]    								
\centering
\begin{tabular}{ r || c | c | c | c | c | c | c | c}
$k$ & 2 & 4 & 6 & 8 & 10 & 12 & $\hdots$ & 100   \\
\hline
\# Clusters & 18 & 218 & 1233 & 1784 & 1807 & 1812 & $\hdots$ &1812 \\
\hline
\# Clusters with one element & 5 & 47 & 845 & 1757 & 1802 & 1812 & $\hdots$ & 1812 \\
\end{tabular}
\caption{Number of clusters with respect to Newton polynomial $N(A_{60}^6,k)$ for even $2\leq k\leq 100$}
\label{tab:Hist NewPoly}
\end{table}

We applied the above clusterization scheme to $C_n$, $28\leq n\leq 150$ and plotted $n$ against $k^*_{single}$ in Figure \ref{fig:k*_single}. Recall that a (pessimistic)  upper bound for $k^*_{single}$ is the number of vertices in $T_n^6$, i.e. $k_{single}^*\leq m_6=\frac{n}{2}-10$ due to Lemma \ref{lemma:cospectral_equivalency}. However, the good news is that the actual growth rate  of $k^*_{single}$ is logarithmic with $n$. Using MATLAB curve fitting toolbox \cite{Matlab} we get 
\begin{equation*}
k_{single}^*(n)\approx -15.13+7.801\log\left(0.7614n-12 \right),
\end{equation*}
with a coefficient of determination $R^2=0.9499$. 

\begin{figure}[H]
\centering
\subfigure[$n\sim k_{single}^* $]{
\includegraphics[width=1\textwidth]{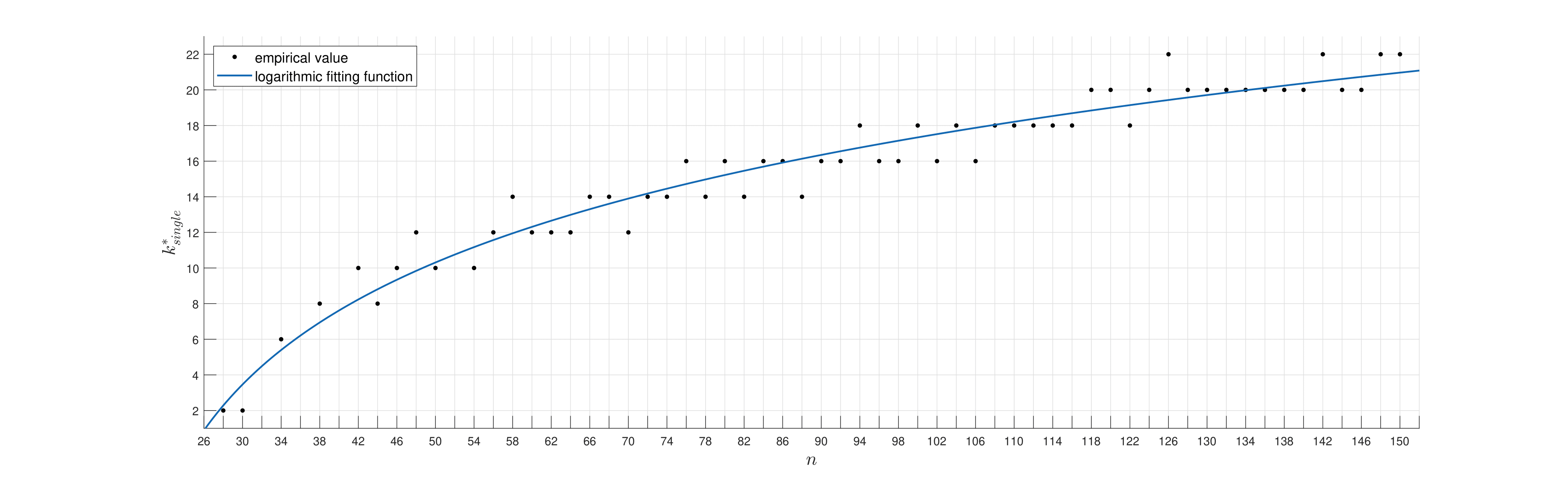}\label{fig:k*_single}
}
\subfigure[$n \sim k_{pair}^*$]{
\includegraphics[width=1\textwidth]{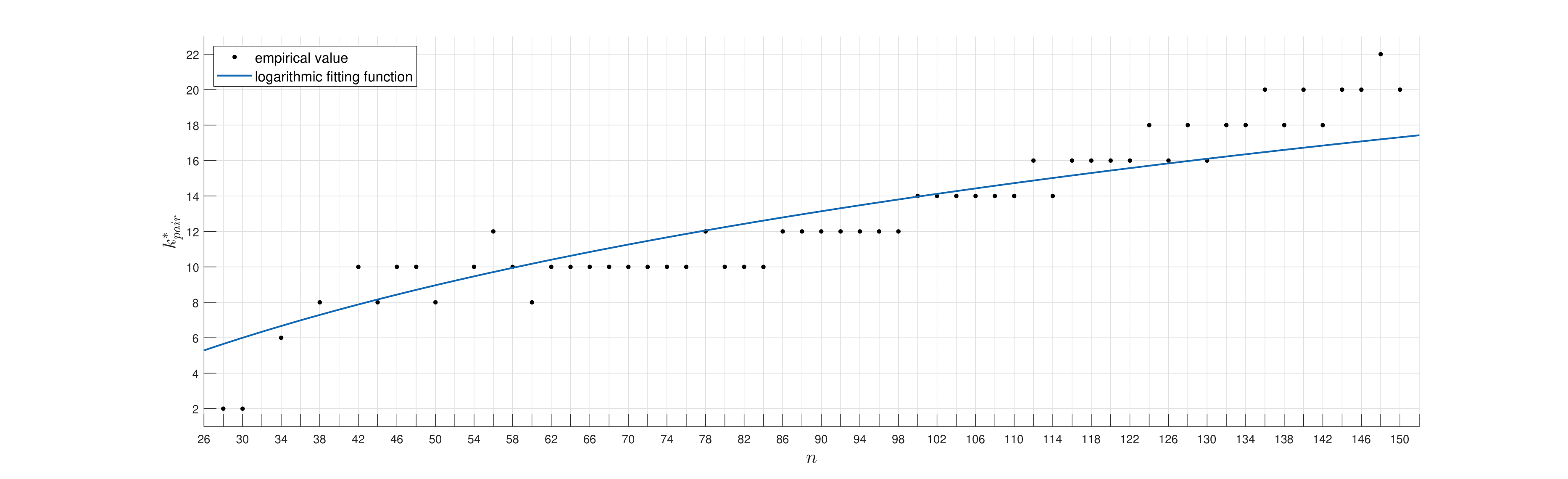}\label{fig:k*_pair}
}
\caption{Minimal degree $k_{single}^*$ (a) and $k_{pair}^*$ (b) needed for the complete clusterization of $C_n$-isomers with $28\leq n \leq 150$}\label{fig:k}
\end{figure}

Next, we use pairs of Newton polynomials $\left(N(A_n^6,k_1),N(A_n^6,k_2)\right)$ with $k_1<k_2\leq k_{single}^*$ in order to cluster all $C_{60}$-isomers. By considering additionally a second Newton polynomial of lower degree, we hope to decrease the needed degree to get a complete clusterization. We define $k_{pair}^*$ as the minimal $k_2$ such that a complete clusterization is given. Indeed, this approach decreases the needed degree significantly (compare both plots in Figure \ref{fig:k}), and therefore,  reduces computational costs. For $C_{60}$ and $T_{60}^6$ the following four tuples with $k_2 <k_{single}^*$ of degrees of Newton polynomials lead to a full classification: 

\begin{align*}
\boldsymbol{k}=(k_1,k_2)\in \big \{ (6,8),  (4,10), (6,10), (8,10) \big \}.
\end{align*}

We get $k_{pair}^*=8$. Next we plotted all values for $k_{pair}^*$ against $n$ and assumed a logarithmic function as for $k_{single}^*$. 
Using MATLAB curve fitting toolbox we get the following approximation 

\begin{align*}
k_{pair}^*(n)&\approx -24.83+10.29\log\left(0.334n+9.989 \right)
\end{align*}
with a coefficient of determination of $R^2=0.8277$.

A third hierarchical approach in order to decrease the needed degree uses a vector with all Newton polynomials up to degree $k\leq k_{single}^*$. Analogously to the first two approaches, we define $k_{hierarchical}^*$ as the minimal $k$ which yields a complete clusterization. This approach decreases e.g. the degree  for $n=78$ from $k_{pair}^*=12$ to $k_{hierarchical}^*=10$. So, this approach does not change the needed degree significantly. Nevertheless, we performed the same interpolation using MATLAB curve fitting Toolbox and got 

\begin{align*}
k_{hierarchical}^*(n)&\approx -97.05+19.83\log\left(1.466n+125.5 \right).
\end{align*}
with $R^2=0.8818$.

In \cite{BBS_2} we discuss another upper bound for $k^*_{single}$ (sharper than $m_6$) using the generalized Stone-Wales operation introduced in \cite{generalized_SW}. In the paper, we give a combinatorial interpretation of $k^*$ and derive an equation system with Newton polynomials which determines whether a fullerene with given $n$ can be constructed.

\subsection{Relative energy}\label{subsection:relative_energy}
A fullerene isomer, which can be chemically separated with a significant mass quantity and uniquely characterized, is called \textit{stable}. In order to decide which $C_{60}$-isomer can be stable the \textit{relative energy} of all of them was calculated with high-accuracy quantum chemistry methods and discussed in \cite{Grimme17}. Here, \textit{relative} means compared with the Buckminster fullerene $C_{60,1812}$ which has the lowest DFT-energy at the $\text{PW6B95-D3}^{\text{ATM}}/\text{def2-QZVP}$ level (cf. \cite{Grimme17}), i.e. Buckminster fullerene has a relative energy of 0. In the sequel, we say that an isomer $C_{n,i}$ is \textit{energetically more stable} than $C_{n,j}$, $i\not= j$, if $C_{n,i}$ has a smaller energy than $C_{n,j}$.

According to \cite{Grimme17} the most stable isomer is $C_{60,1812}$ and the second stable one is $C_{60,1809}$. At the other end of the ranking the three least stable ones are $C_{60,2}$, $C_{60,3}$ and $C_{60,1}$. It has been assumed for a long time that an isomer is the more stable the less adjacent pentagon it has. Indeed, calculations of \cite{Grimme17} allow the conclusion that each pair of two adjacent pentagons leads to a increase in the relative energy of an isomer of about 20 to 25 kcal $\text{mol}^{-1}$. The amount of such pentagon pairs can be described with Fowler-Manolopoulos pentagon indices $p_i:=\# \{\mbox{pentagons which are adjacent to } i \mbox{ other pentagons} \}$, such that the sum of $p_1$ up to $p_5$ is equal to 12 for every fullerene, cf. \cite{FowlerManop}. Based on these values, the pentagon signature $P_1=1/2 \sum_{i=1}^5ip_i$ can be calculated, which quantifies the amount of connected pentagons. Clustering all $C_{60}$-isomers according to the pentagon signature, five isomers stand out, namely $C_{60,1812} (P_1=0)$, $C_{60,1809} (P_1=2)$, $C_{60,2} (P_1=16)$, $C_{60,3} (P_1=18)$ and $C_{60,1} (P_1=20)$. The signature can be easily read from Figures \ref{fig:C_60,1}-\ref{fig:C_60,1812}. Pentagons signatures of the remaining isomers lie between $2$ and $16$. For each of these values, at least two isomers exist with the same pentagon signature. By Lemma \ref{lemm:Newton} \ref{Lemma_1_b}) and \ref{lemma:edges_T6-T5} we get the following 
\begin{proposition}
For any $C_n$-isomer it holds
\begin{equation*}
P_1=\frac{N(A_n^6,2)}{2}-\frac{3n}{2}+60.
\end{equation*}
\end{proposition}

In Table \ref{tab:Histo_NP&PS}, $C_{60}$-isomers are listed in the same order given by their relative energy, by their pentagon signature and their Newton polynomial of degree 2.  

\begin{table}[H]    								
\centering
\begin{tabular}{ c || c | c | c | c | c }
Isomer & $C_{60,1812}$ & $C_{60,1809}$ & $C_{60,2}$ & $C_{60,3}$ & $C_{60,1}$ \\
\hline
$N(A^6_{60},2)$ & 60 & 64 & 92 &  96 & 100 \\
\hline
$P_1$ & 0 & 2 & 16 & 18 & 20 
\end{tabular}
\caption{Five $C_{60}$-isomers with unique Newton polynomial $N(A^6_{60},2)$ and Pentagon signature $P_1$ sorted by their relative energy in ascending order.}
\label{tab:Histo_NP&PS}
\end{table}

Indeed, one gets more information about a fullerene structure looking on hexagons than on pentagons. This becomes clear looking at fullerenes with large $n$. For example, $C_{80}$ has 7 IPR-isomers, so their pentagon structure and pentagon signature are the same. Nevertheless, only two of them have been produced in pure form, although DFT calculations have been done for all of them, see \cite{C80}. As a result of \cite{C80}, only two IPR--isomers can be claimed stable. Hence, all descriptors based on the pentagon structure do not properly predict stability.

In \cite{Grimme17} a \textit{good stability criterion} is defined as the one which can identify $C_{60,1812}$ and $C_{60,1809}$ as the most stable and $C_{60,2}, C_{60,3}$ and $C_{60,1}$ as the least stable isomers in the correct energetic order. Additionally, the Pearson coefficient $\rho$  of linear correlation between the relative energies of all  $C_{60}$--isomers and their criterion values should be larger than $0.6$. Finally, the slope and the Pearson correlation coefficient in the linear regression of relative energy vs. the criterion for $C_{60}$--isomers with $P_1\in\lbrace4,\ldots, 14 \rbrace$ should have the same sign.

As we have seen in Table \ref{tab:Histo_NP&PS}, Newton polynomials yield the correct order of the most and least stable $C_{60}$-isomers. Next, we perform a linear regression (using MATLAB curve fitting Toolbox) of $N(A_{60}^6, k)$ vs. relative energies of all $C_{60}$ isomers for all even $4\le k\le 12$. For the case $k=2$ Newton polynomial is equivalent to the 1st moment hexagon Signature $H_1$, which is listed in \cite[Table 3]{Grimme17} as a good stability criterion. Hence, $N(A_{60}^6,2)$ is a good stability criterion as well and can be neglected in further considerations. Table \ref{tab:C40C60C80} shows that Pearson correlation coefficient $\rho$  is much higher than $0.6$ for all considered $k$. For degrees $k=8,10,12$, Newton polynomials $N(A_{60}^6, k)$ get very large, and therefore we took a logarithmic scale. But even with linear scale, one gets correlation coefficients larger than $0.6$ in these three cases, compare Table \ref{tab:app_lin_reg}.  

Next we divided all isomers of $C_{60}$ into $18$ subsets $G_i=\lbrace P\in C_{60} \mid P_1(P)=i \rbrace$ according to their pentagon signature $i\in\lbrace 0,2,3,\ldots,16,18,20\rbrace$. Then we performed a linear regression of Newton polynomials of different degrees vs. relative energies of isomers in $G_i$ for every $i\notin\lbrace 0,2,3,15,16, 18, 20 \rbrace$ as it is required in \cite{Grimme17}. Our results are listed in Appendix, Table \ref{tab:app_linear_regre_2}. For $k\in\lbrace 4,6\rbrace$ and $i\in\lbrace4,5\rbrace$ we get Pearson correlation coefficients and slopes with a negative sign, unlike for all other combinations of $k$ and $i$. This can be explained by the fact that $G_4$ and $G_5$ do not contain many isomers. More precisely, $|G_4|=17,$ and $|G_5|=86$ holds. So, neglecting these two cases would yield that $N(A_{60}^6,k)$ with $k=4,6$ is a good stability criterion. 

For $k=10,12$ one gets positive slopes and Pearson correlation coefficients in all cases, and therefore Newton polynomials of degree 10 and 12, in particular of degree $k_{single}^*$, entirely fulfil all conditions of a good stability criterion.

\begin{table}[H]    								
\centering
\begin{tabular}{c | c | c }
Linear regression  &  $\rho$ &slope  \\ \hline
$N(A_{60}^6,4)\sim \text{relative energy}$ &  0.95 & 0.45 \\
$N(A_{60}^6,6)\sim \text{relative energy}$ &  0.95 & 0.02 \\
$log\left(N(A_{60}^6,8)\right)\sim \text{relative energy}$ &  0.945 & 106.6  \\
$log\left(N(A_{60}^6,10)\right)\sim \text{relative energy}$ &  0.94 & 81.36 \\
$log\left(N(A_{60}^6,12)\right)\sim \text{relative energy}$ &  0.94 & 66.1 
\end{tabular}
\caption{Pearson correlation coefficient $\rho$ and the slope of linear regression between Newton polynomials and the relative energy of all $C_{60}$--isomers given in \cite{Grimme17}.}
\label{tab:C40C60C80}
\end{table}

To check whether Newton polynomials can distinguish between IPR-isomers, i.e. yield their energetically correct order, we computed Newton polynomials of all 31924 $C_{80}$-isomers. Within the whole set of $C_{80}$ the seven IPR-isomers have the smallest Newton polynomials. But ordering the set of IPR-isomers according to Newton polynomials leads to the observation that the most stable IPR-isomers have the greatest Newton polynomials. These seven isomers are listed in Table \ref{tab:C80IPR}. So, it seems that with increasing $n$ Newton polynomials $N(A_{n}^6,k)$ for even $k\geq 2$ remain a good stability criterion. 

\begin{table}[H]    								
\centering
\begin{tabular}{c | c | c | c | c | c | c | c | c }
Isomer & rel. Energy  & tr$(A_{80,6}^4)$ & tr$(A_{80,6}^6)$ & tr$(A_{80,6}^8)$ & tr$(A_{80,6}^{10})$ & tr$(A_{80,6}^{12})$ & tr$(A_{80,6}^{14})$ & tr$(A_{80,6}^{16})$  \\ \hline
31918 & 0 &  1040 & 12960 & $19.4\times 10^{4}$ & $31.7\times 10^5$ & $5.4\times 10^7$ & $9.4\times 10^8$ & $1.7\times 10^{10}$\\
31919& 0.41   & 1016 & 12144 & $17.3\times 10^4$ & $26.9\times 10^5$ & $4.4\times 10^7$ & $7.2\times 10^8$ & $1.2\times 10^{10}$ \\
31920 & 2.58  & 960 & 10530 & $13.7\times 10^4$ & $19.4\times 10^5$ & $2.9\times 10^7$ & $4.5\times 10^8$ & $0.7\times 10^{10}$\\
31921 & 4.37  & 984 & 11442 & $16\times 10^4$ & $24.3\times 10^5$ & $3.9\times 10^7$ & $6.3\times 10^8$ & $1.1\times 10^{10}$ \\
31922 & 1.48  & 920 & 9732 & $12.2\times 10^4$ & $17\times 10^5$ & $2.5\times 10^7$ & $3.7\times 10^8$ & $0.6\times 10^{10}$ \\
31923 & 3.32 &  880 & $8940$ & $10.9\times 10^4$ & $14.8\times 10^5$ & $2.1\times 10^7$ & $3.1\times 10^8$ & $0.48\times 10^{10}$ \\
31924 & 14.31 & 840 & 8520 & $10.5\times 10^4$ & $14.4\times 10^5$ & $2.1\times 10^7$  & $3.1\times 10^8$ & $0.47\times 10^{10}$
\end{tabular}
\caption{Seven IPR-isomers of $C_{80}$, their relative energy in kcal/mol$^{-1}$  and Newton polynomials. Only the isomers 31918 and 31919 can be chemically separated so far, cf. \cite{C80}.}
\label{tab:C80IPR}
\end{table}

\subsection{Asymmetry coefficients of isomers of $\boldsymbol{C_{60}}$}
\label{subsec: Spectral properties of fullerenes} 
The Fowler asymmetry parameter is claimed to be a good stability criterion \cite{Grimme17}. Check whether the asymmetry coefficient $\theta$  defined in Section \ref{sec:Spectral} is a good stability criterion as well. The asymmetry coefficients of isomers shown in Figure \ref{fig:schlegel+dual} are $\theta_1=1, ~\theta_2=1.4,~\theta_3=1.2,~\theta_{1809}=0.8$ and $\theta_{1812}=0$. The histogram of $\theta$ of all $C_{60}$-isomers is shown in Figure \ref{fig:asy_coef}.
\begin{figure}[H]
    \centering
    \includegraphics[scale=0.2]{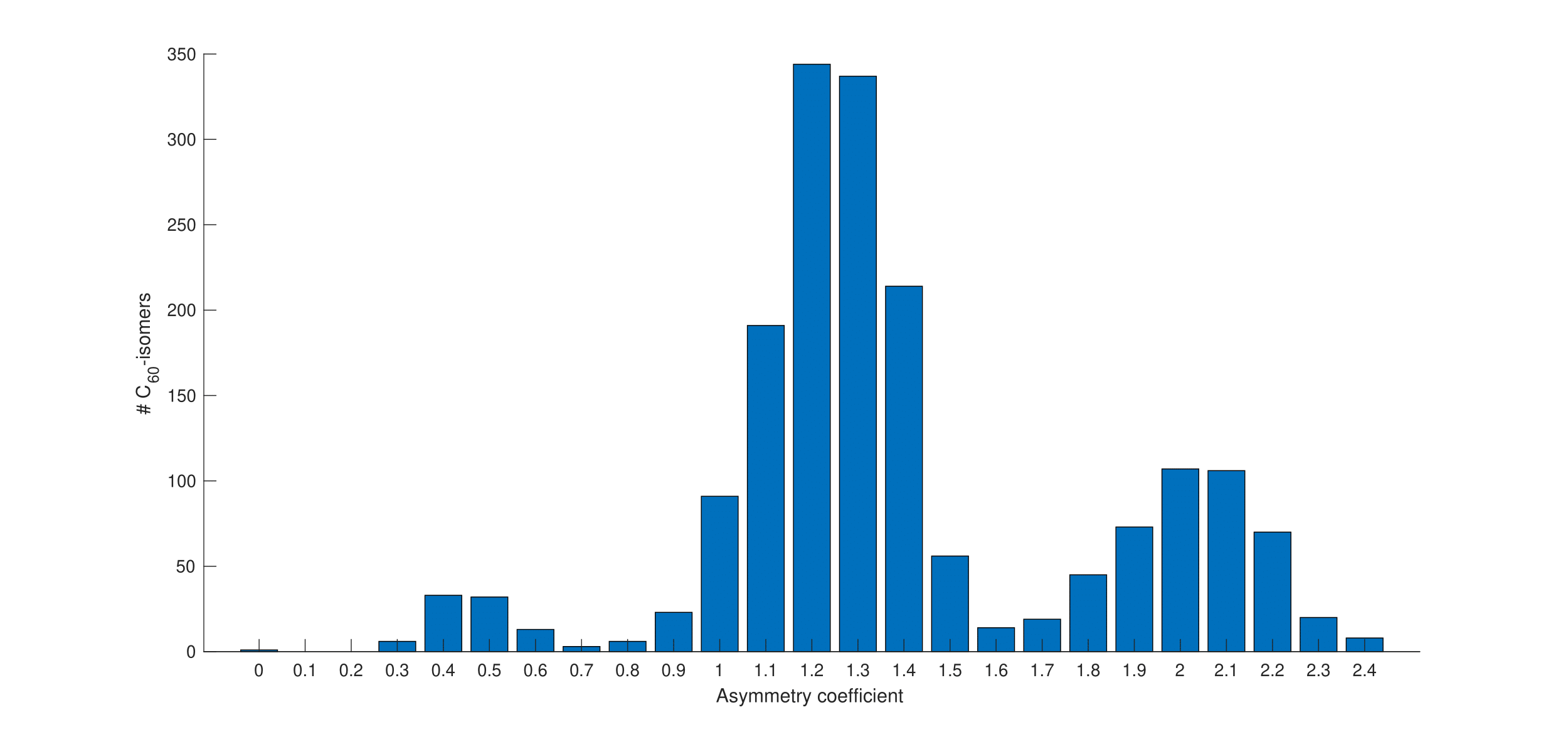}
    \caption{Histogram of 23 different asymmetry coefficients of $C_{60}$-isomers}
    \label{fig:asy_coef}
\end{figure}

It turns out that the asymmetry coefficient $\theta$ is not a good stability criterion since it does not preserve the energetic order required in \cite{Grimme17}. Thus, the asymmetry coefficient of eight isomers is equal 2.4, which is the largest value. These isomers are $C_{60,1334},C_{60,1554},C_{60,1676},$ $C_{60,1740},C_{60,1741},C_{60,1742},C_{60,1761}$ and $C_{60,1784}$.
Figure \ref{fig:asy_coef_example} shows $C_{60,1784}$, which has six hexagons with valency  two, twelve hexagons with valency four and two hexagons with valency six. So, the mean valency is 3.6, the maximal valency is 6 and the resulting asymmetry coefficient equals 2.4. 

\begin{figure}[H]
	\centering
	\includegraphics[scale=0.2]{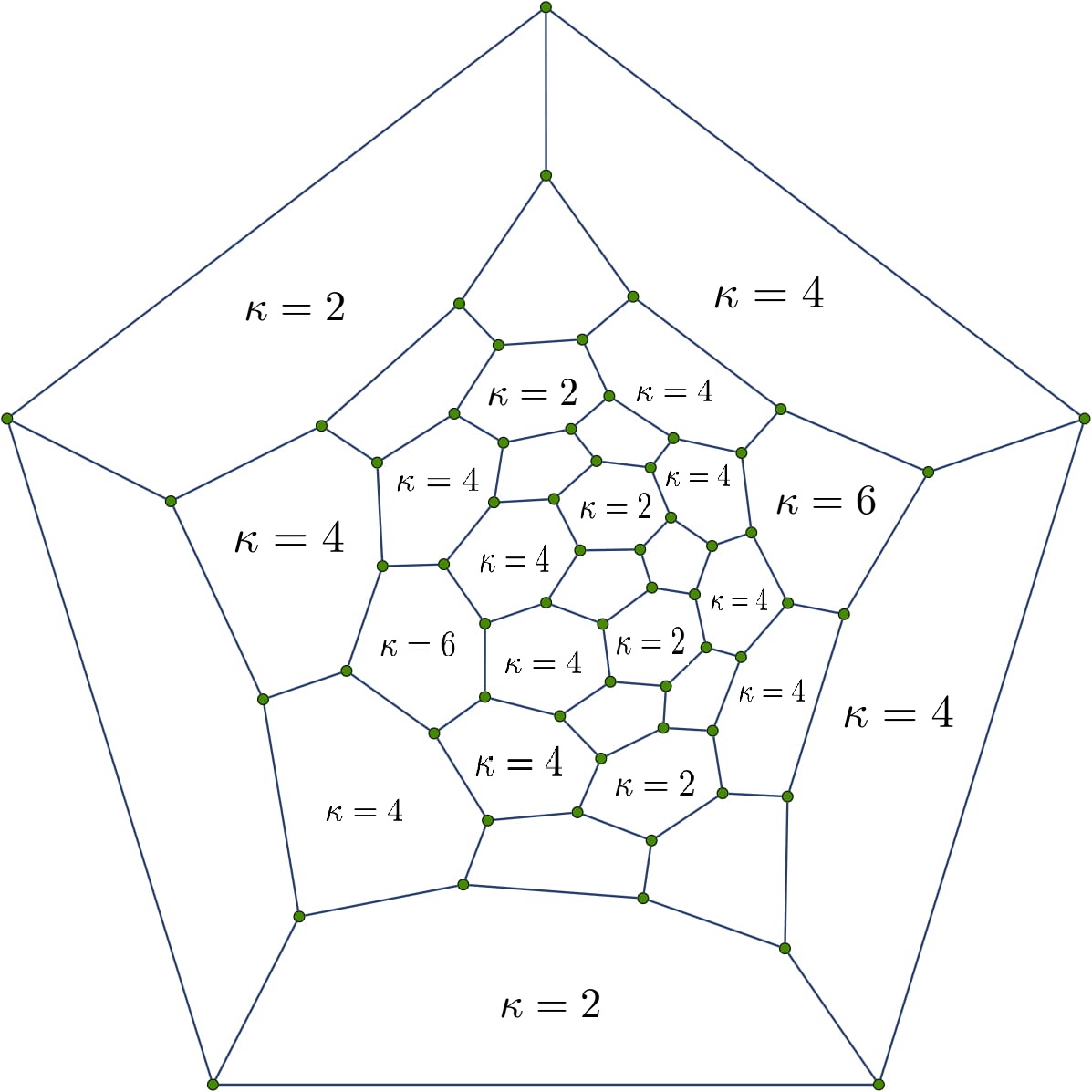}
	\caption{Schlegel diagram of one of the eight isomers with largest asymmetry coefficient $\theta=2.4$}
	\label{fig:asy_coef_example}
\end{figure} 

In Table \ref{tab:theta_relativeEnergy}, asymmetry coefficients and energetic order numbers  (1 = most stable, 1812 = least stable) \cite{Grimme17} are listed for some outstanding isomers of $C_{60}$, i.e. the ones with smallest and largest $\theta$ as well as those shown in Figure \ref{fig:schlegel+dual}. One can see that the three least stable as well as the two most stable isomers have the asymmetry coefficient less than $2.4$. In addition, the relative stability of the eight isomers with $\theta=2.4$ varies from 313 to 1122. This leads to the conclusion that $\theta$ is not a good stability predictor.   

\begin{table}[H]    								
	\centering
	\begin{tabular}{ r || c | c | c | c | c | c | c | c | c | c | c | c | c}
		$i$ & 1 & 2 & 3 & 1334 & 1554 & 1676 & 1740 & 1741 & 1742 & 1761 & 1784 & 1809 & 1812 \\
		\hline
		$\theta_i$ & 1 & 1.4 & 1.2 & 2.4 & 2.4 & 2.4 & 2.4 & 2.4 & 2.4 & 2.4 & 2.4 & 0.8 & 0 \\
		\hline
	 stability& 1812 & 1810 & 1811 & 313 & 472 & 326 & 576 & 555 & 367 & 727 & 1122 & 2 & 1 \\
	\end{tabular}
	\caption{Comparison of the asymmetry coefficient $\theta$ and energetic stability of some $C_{60}$--isomers}
	\label{tab:theta_relativeEnergy}
\end{table}

\section{Conclusion}
\label{sec: Conclusion}
We present an easy to compute functional of spectra of the graphs $T_{n}$ and $T_n^6$ which classifies all $C_{n}$-isomers. Thereby we focus on the structure of the dual graph of hexagonal facets of $C_n$ and its adjacency matrix $A_{n}^6$.\\ 
The spectra of the adjacency matrices are characteristic to combinatorial isomers described above. It becomes apparent that the Newton polynomial of degree $2,10,12(=k_{single}^*)$ of $T_{n}^6$ appears to be a good stability criterion. So, Newton polynomials of $T_{60}^6$ can be added to the list presented in \cite[Table 3]{Grimme17} as indices, which, depending on the degree, fulfil the criteria partly or entirely. We show that Newton polynomials generalize the Pentagon signature and better describe the fullerene structure. The interpretation of these Newton polynomials is very easy for $k\leq 3$, but gets demanding with increasing $k$. 
 
\section*{Acknowledgements}
We express our gratitude to Markus Schandar who was involved in programming of spectra of fullerenes at the early stage of this research. We would also like to thank Max von Delius and Konstantin Amsharov for discussions on the chemistry of fullerenes. We are indebted to Axel Gro\ss  \ for the reference \cite{Grimme17}  and his lectures on the DFT method.
 
\newpage
\bibliography{Literatur}{}
\bibliographystyle{abbrv}

\newpage
\section*{Appendix}


\subsection*{Matlab - Curve Fitting Toolbox - Linear Regression}

\begin{table}[H]
\centering
\begin{tabular}[t]{l | c | c | c }
Independent variable & $\rho$ & slope & intercept \\
\hline
$N(A_{60}^6,2)$ & 0.9524 & 11.6996  & -701.3801 \\
$N(A_{60}^6,4)$ & 0.9557 & 0.4501 & -96.2619 \\
$N(A_{60}^6,6)$ & 0.9514 & 0.0201 & 37.6084\\
$N(A_{60}^6,8)$ & 0.9328 & $9.135\cdot 10^{-4}$ & 91.6962\\
$N(A_{60}^6,10)$ & 0.8974 & $4.0357\cdot 10^{-5}$ & 122.1883\\
$N(A_{60}^6,12)$ & 0.8456 & $1.7054\cdot 10^{-6}$ & 142.6633 \\
$log\left(N(A_{60}^6,8)\right)$ & 0.9452 & 106.6054 & $-1.0359\cdot 10^{3}$\\
$log\left(N(A_{60}^6,10)\right)$ & 0.9421 & 81.3552 & -964.0464 \\
$log\left(N(A_{60}^6,12)\right)$ & 0.9392 & 66.1011 & -927.8387 \\
\end{tabular}\caption{Coefficients of linear regression \textit{Independent variable $\sim$ relative energy} for all $C_{60}$-isomers.}\label{tab:app_lin_reg}
\end{table}
\begin{table}[H]
\centering
\begin{tabular}[t]{l | c | c | c | c||}
Ind. var. & $i$  & $\rho$ & slope & intercept\\
\hline
$N(A_{60}^6,4)$ & $G_4$ & -0.4059  & -0.2933 & 217.0715\\
$N(A_{60}^6,4)$ & $G_5$ & -0.1603  & -0.0701 & 149.9463\\
$N(A_{60}^6,4)$ & $G_6$ & 0.2720  & 0.1371 & 68.5874\\
$N(A_{60}^6,4)$ & $G_7$ & 0.3833 & 0.2439 & 23.3254\\
$N(A_{60}^6,4)$ & $G_8$ & 0.4208 & 0.3077 & -5.3254\\
$N(A_{60}^6,4)$ & $G_9$ & 0.49 & 0.3954 & -582102\\
$N(A_{60}^6,4)$ & $G_{10}$ & 0.4059 & 0.3617 & -30.1859\\
$N(A_{60}^6,4)$ & $G_{11}$  & 0.3326 & 0.3337 & -5.548\\
$N(A_{60}^6,4)$ & $G_{12}$ & 0.2198 & 0.0989 & 195.1829\\
$N(A_{60}^6,4)$ & $G_{13}$ & 0.4858 & 0.2748 & 53.9504\\
$N(A_{60}^6,4)$ & $G_{14}$ & 0.3907 & 0.1841 & 155.4456\\
$N(A_{60}^6,6)$ & $G_4$ & -0.0529 & -0.0023 & 98.0643\\
$N(A_{60}^6,6)$ & $G_5$ & -0.0513 & -0.001 & 120.9228\\
$N(A_{60}^6,6)$ & $G_6$ & 0.3469 & 0.0072 & 102.8111\\
$N(A_{60}^6,6)$ & $G_7$ & 0.4957 & 0.0121 & 88.6426\\
$N(A_{60}^6,6)$ & $G_8$ & 0.4991 & 0.0129 & 94.0273\\
$N(A_{60}^6,6)$ & $G_9$ & 0.5638 & 0.0151 & 83.1361\\
$N(A_{60}^6,6)$ & $G_{10}$ & 0.4928 & 0.014 & 99.8652\\
$N(A_{60}^6,6)$ & $G_{11}$ & 0.4525 & 0.0129 & 117.1669\\
$N(A_{60}^6,6)$ & $G_{12}$ & 0.2435 & 0.0032 & 239.4308\\
$N(A_{60}^6,6)$ & $G_{13}$ & 0.4746 & 0.0075 & 196.7939\\
$N(A_{60}^6,6)$ & $G_{14}$ & 0.3256 & 0.0043 & 264.7982\\
$N(A_{60}^6,8)$ & $G_4$ & 0.138 & $3.7756\cdot 10^{-4}$  & 76.2233 \\
$N(A_{60}^6,8)$ & $G_5$ & -0.0016 & $-1.8288\cdot 10^{-6}$ & 116.4303\\
$N(A_{60}^6,8)$ & $G_6$ & 0.3664 & $3.973 \cdot 10^{-4}$ & 116.0842 \\
$N(A_{60}^6,8)$ & $G_7$ & 0.5148 & $6.2015 \cdot 10^{-4}$ & 114.8225 \\
$N(A_{60}^6,8)$ & $G_8$ & 0.495 & $5.8553\cdot 10^{-4}$ & 129.7291 \\
$N(A_{60}^6,8)$ & $G_9$ & 0.5565 & $6.5117\cdot 10^{-4}$ & 130.8349 \\
$N(A_{60}^6,8)$ & $G_{10}$ & 0.4894 & $5.7816 \cdot 10^{-4}$ & 148.9877 \\
$N(A_{60}^6,8)$ & $G_{11}$ & 0.4509 & $4.869 \cdot 10^{-4}$ & 171.0159 \\
$N(A_{60}^6,8)$ & $G_{12}$ & 0.219 & $1.0241 \cdot 10^{-4}$ & 256.646 \\
$N(A_{60}^6,8)$ & $G_{13}$ & 0.4498 & $2.4936 \cdot 10^{-4}$ & 237.326 \\
$N(A_{60}^6,8)$ & $G_{14}$ & 0.2674 & $1.2357 \cdot 10^{-4}$ & 294.4285 \\
\end{tabular}\begin{tabular}[t]{l | c | c | c | c}
Ind. var. & $i$ & $\rho$ & slope & intercept\\
\hline
$N(A_{60}^6,10)$ & $G_4$ & 0.2237 & $4.0346\cdot 10^{-5}$ & 74.2129 \\
$N(A_{60}^6,10)$ & $G_5$ & 0.0129 & $8.9429\cdot 10^{-7}$ & 115.8252 \\
$N(A_{60}^6,10)$ & $G_6$ & 0.3647 & $2.2609 \cdot 10^{-5}$ & 122.7497 \\
$N(A_{60}^6,10)$ & $G_7$ & 0.4987 & $3.2803\cdot 10^{-5}$ & 128.2062 \\
$N(A_{60}^6,10)$ & $G_8$ & 0.4598 & $2.7315\cdot 10^{-5}$ & 148.0877 \\
$N(A_{60}^6,10)$ & $G_9$ & 0.5183 & $2.9027\cdot 10^{-5}$ & 154.8938 \\
$N(A_{60}^6,10)$ & $G_{10}$ & 0.4507 & $2.4108 \cdot 10^{-5}$ & 175.3146 \\
$N(A_{60}^6,10)$ & $G_{11}$ & 0.4088 & $1.8534 \cdot 10^{-5}$ & 199.3648 \\
$N(A_{60}^6,10)$ & $G_{12}$ & 0.1805 & $3.249 \cdot 10^{-6}$ & 265.5572 \\
$N(A_{60}^6,10)$ & $G_{13}$ & 0.4187 & $8.8169\cdot 10^{-6}$ & 256.1029 \\
$N(A_{60}^6,10)$ & $G_{14}$ & 0.221 & $3.7874 \cdot 10^{-6}$ & 307.592 \\
$N(A_{60}^6,12)$ & $G_4$ & 0.2632 & $3.2603\cdot 10^{-6}$ & 75.4691 \\
$N(A_{60}^6,12)$ & $G_5$ & 0.0111 & $4.8793\cdot 10^{-8}$ & 115.988 \\
$N(A_{60}^6,12)$ & $G_6$ & 0.3552 & $1.3141\cdot 10^{-6}$ & 126.7007 \\
$N(A_{60}^6,12)$ & $G_7$ & 0.4688 & $1.7586\cdot 10^{-6}$ & 136.4829 \\
$N(A_{60}^6,12)$ & $G_8$ & 0.4143 & $1.2821 \cdot 10^{-6}$ & 159.4621 \\
$N(A_{60}^6,12)$ & $G_9$ & 0.4696 & $1.307\cdot 10^{-6}$ & 169.7628 \\
$N(A_{60}^6,12)$ & $G_{10}$ & 0.4011 & $1.0054\cdot 10^{-6}$ & 191.9619 \\
$N(A_{60}^6,12)$ & $G_{11}$ & 0.3571 & $7.0982\cdot 10^{-7}$ & 216.4457 \\
$N(A_{60}^6,12)$ & $G_{12}$ & 0.1443 & $1.0342\cdot 10^{-7}$ & 270.5091 \\
$N(A_{60}^6,12)$ & $G_{13}$ & 0.3907 & $3.2541\cdot 10^{-7}$ & 266.6058 \\
$N(A_{60}^6,12)$ & $G_{14}$ & 0.1884 & $1.2319\cdot 10^{-7}$ & 314.4527 \\
\end{tabular}\caption{Linear regression of an independent variable $N(A_{60}^6,k)$, $4\le k\le 12$ even, vs. relative energy over the subsets $G_i$ of $C_{60}$--isomers}\label{tab:app_linear_regre_2}
\end{table}

\end{document}